\documentclass[10pt, oneside]{amsart}   	
\usepackage[margin=1in]{geometry}                		
\usepackage{cite}
\usepackage{graphicx}				
\usepackage{amssymb}

\usepackage{amssymb,bbm}
\usepackage{amsthm}
\usepackage{graphicx}
\usepackage{mathtools}
\usepackage{amsmath,amsfonts}
\usepackage{epsfig}
\usepackage{comment}
\usepackage{color}
\usepackage{amsmath}

\usepackage{todonotes}

\usepackage{float}
\usepackage{tikz}
\usetikzlibrary{shapes.geometric}

\newcommand{\adm}{{\rm Adm}}

\newcommand{\Mod}{{\rm Mod}}
\newcommand{\Cp}{{\rm Cap}}
\newcommand{\len}{{\rm Len}}
\newcommand{\image}{{\rm Image}}
\newcommand{\gap}{{\rm D}}
\newcommand{\mesh}{{\rm Mesh}}
\newcommand{\ones}{{\mathbbm 1}}

\newcommand{\Z}{{\mathbb Z}}
\newcommand{\R}{{\mathbb R}}
\newcommand{\bP}{{\mathbf{P}}}

\newcommand{\N}{{\mathbb N}}

\newcommand{\de}{\delta}
\newcommand{\ep}{\epsilon}

\newcommand{\ga}{\gamma}

\newcommand{\Ga}{\Gamma}

\newcommand{\defeq}{:=}

\DeclareMathOperator{\length}{length}
\DeclareMathOperator{\diam}{diam}

\DeclareMathOperator{\lip}{lip}

\def\vint_#1{\mathchoice
          {\mathop{\vrule width 6pt height 3 pt depth -2.5pt
                  \kern -8pt \intop}\nolimits_{#1}}%
          {\mathop{\vrule width 5pt height 3 pt depth -2.6pt
                  \kern -6pt \intop}\nolimits_{#1}}%
          {\mathop{\vrule width 5pt height 3 pt depth -2.6pt
                  \kern -6pt \intop}\nolimits_{#1}}%
          {\mathop{\vrule width 5pt height 3 pt depth -2.6pt
                  \kern -6pt \intop}\nolimits_{#1}}}

\theoremstyle{plain}
\newtheorem{theorem}[equation]{Theorem}
\newtheorem{corollary}[equation]{Corollary}
\newtheorem{lemma}[equation]{Lemma}

\newtheorem{proposition}[equation]{Proposition}

\theoremstyle{definition}
\newtheorem{definition}[equation]{Definition}

\newtheorem{example}[equation]{Example}
\newtheorem{remark}[equation]{Remark}

 \numberwithin{equation}{section}

 \newcommand{\bi}{\begin{itemize}}
\newcommand{\ei}{\end{itemize}}

\title{Density of continuous functions in Sobolev spaces with applications to capacity}
\author{Sylvester Eriksson-Bique}
\address{Department of Mathematics and Statistics, University of Jyv\"askyl\"a,
Seminaarinkatu 15, PO Box 35,  FI-40014 University of Jyv\"askyl\"a, Finland}
\email{sylvester.d.eriksson-bique@jyu.fi}

\author{Pietro Poggi-Corradini}
\address{Kansas State University,
Department of Mathematics,
138 Cardwell Hall,
Manhattan, KS 66506}
\email{pietro@math.ksu.edu}

\thanks{The first author was partially supported by Finnish Academy Grants n.~345005 and n.~356861. The second author is partially supported by NSF DMS n.~2154032. The authors thank Nages Shanmugalingam, and Anders and Jana Bj$\ddot{\text{o}}$rn for useful conversations on these topics. The paper benefited from helpful comments by an anonymous referee.}

\begin{document}
\maketitle

\begin{abstract}
We show that capacity can be computed with locally Lipschitz functions in locally complete and separable metric spaces. Further, we show that if $(X,d,\mu)$ is a locally complete and separable metric measure space, then continuous functions are dense in the Newtonian space $N^{1,p}(X)$. Here the measure $\mu$ is Borel and is finite and positive on all metric balls. In particular, we don't assume properness of $X$, doubling of $\mu$ or any Poincar\'e inequalities. These resolve, partially or fully, questions posed by a number of authors, including J.~Heinonen, A.~Bj\"orn and J.~Bj\"orn.  In contrast to much of the past work, our results apply to \emph{locally complete} spaces $X$ and dispenses with the frequently used regularity assumptions: doubling, properness, Poincar\'e inequality, Loewner property or quasiconvexity. 
\end{abstract}

\section{Introduction}

Solutions to variational problems on metric measure spaces $(X,d,\mu)$, such as $p$-harmonic functions, may fail to be continuous or Lipschitz in a fully general setting. However, a useful tool is to approximate such minimizers by continuous or Lipschitz functions.
In many works, see for instance \cite{che99,lahtisha, shabook, bjornbook},
 one places assumptions such as, the doubling property, the Poincar\'e inequality, or properness, to prove density of Lipschitz functions. Doubling and Poincar\'e inequalities are natural in certain settings, such as $A_p$-weighted spaces \cite{HKM06}, Carnot groups \cite{le2018primer}, boundaries of certain hyperbolic groups \cite{bourdonpi} and manifolds with Ricci bounds \cite{cheeger1997structure,guofangwei}. 
 However, there are many important settings where these assumptions are overly restrictive, and we name just a handful of such: studying generalized notions of scalar curvature and intrinsic limits of manifolds with scalar curvature bounds \cite{gromov2014plateau,sormani2011intrinsic}, studying integral currents in metric spaces \cite{currents},  metric manifolds and uniformization of metric surfaces \cite{rajala,basso2023geometric,ntalampekos2023polyhedral}, the study of analysis on fractals \cite{capogna2022neumann}, Sobolev spaces on infinite dimensinal spaces such as the Wasserstein space \cite{sodini}, spaces equipped with more general weights \cite{ambrosio}, or complete and rectifiable spaces \cite{bate}.  In all these cases Sobolev spaces, and associated differential structures, still play a crucial role. 
 
 Our contribution in this paper is to remove the assumptions of doubling and Poincar\'e and to replace these with a much weaker local completeness assumption, and to still prove three fundamental properties: the density of continuous functions in the Sobolev space, equivalence of different notions of capacity and the property that the Sobolev capacity is a Choquet capacity. This  clarifies substantially the methods and dependence on assumptions. Further, it requires the use of new approximation and extension methods.

The variational problems we consider are classical in metric measure spaces, and arise from the definition of an  \emph{upper gradient}. A Borel function $g:X\to [0,\infty]$ is said to be an upper gradient for a function $u:X \to [-\infty,\infty]$ if for all rectifiable curves $\gamma:[0,1]\to X$, we have
\begin{equation}\label{eq:ug}
|u(\gamma(0))-u(\gamma(1))| \leq \int g ds,
\end{equation}
where we interpret $|\infty-\infty|=\infty$.

Next, if $1\le p<\infty$, consider the condenser $p$-capacity problem
\begin{equation}\label{eq:p-capacity}
\Cp_p(E,F) := \inf_{\stackrel{u|_E = 0}{u|_F = 1}} \int g^p ~d\mu,
\end{equation}
where $E$ and $F$ are disjoint closed sets in $X$, and where $u$ is taken to be a function on $X$ and $g$ is an upper gradient for $u$.
Heinonen and Koskela asked in \cite[Remark 2.13]{heinonenkoskela}, if the infimum in this definition could be taken over functions $u$ that are  continuous or locally Lipschitz. In general, we may consider a collection $W$ of pairs $(u,g)$, where  $u$ is a function on $X$ and $g$ is an upper gradient for $u$, and define
\[
\Cp_p^W(E,F):=\inf_{\stackrel{u|_E = 0, u|_F = 1}{(u,g) \in W}} \int g^p ~d\mu.
\]
By varying $W$, we obtain different versions of capacity considered in the literature. We focus on three variants which have appeared in the literature: i) $W={\rm lip}$ corresponds to all pairs $(u,g)$ where $u$ is locally Lipschitz, ii) $W={\rm cont}$ is the collection of all pairs $(u,g)$ where $u$ is continuous, iii) and $W=(\lip,\lip)$ is the collection of all pairs $(u,g)$ where $u$  and $g$ are locally Lipschitz. It is trivial, that restricting to the collections to $(\lip,\lip), {\rm lip}$ or ${\rm cont}$ produces a capacity, which is larger than the unrestricted capacity in (\ref{eq:p-capacity}). The problem, which bears a close affinity to approximation, is to show that these restricted capacities are still equal to the unrestricted capacity. Throughout, we will assume that $\mu$ is a Borel measure on $X$, which is positive and finite on all balls, that is $0<\mu(B)<\infty$ for each $B=B(x,r)$, with $x\in X$ and $r>0$.

Our first main result is the following.
\begin{theorem}\label{thm:capacity}
Let $(X,d,\mu)$ be a locally complete and separable metric measure space and let $E,F \subset X$ be two closed, non-empty disjoint sets with $d(E,F)>0$. Then, for $p\in[1,\infty)$,
\begin{equation}\label{eq:cap-equality}
\Cp_p(E,F)=\Cp_p^{\rm cont}(E,F)=\Cp_p^{\lip}(E,F)=\Cp_p^{(\lip,\lip)}(E,F).
\end{equation}
\end{theorem} 
\begin{remark}\label{rmk:modulus} If $\Gamma(E,F)$ is the family of rectifiable curves connecting $E$ to $F$, then one has the equality between modulus and capacity $\Cp_p(E,F)=\Mod_p(E,F)$. Whenever $(u,g)$ is admissible for the capacity, $g$ is admissible for the modulus. Conversely, if $g$ is admissible for the modulus, then there exists a $u$ which is admissible for the capacity so that $g$ is an upper gradient of $u$. Indeed, such a $u$ is obtained by ``integrating'' $g$. See Section \ref{sec:mod-sob} for the definition of modulus, and \cite[Proposition 2.17]{heinonenkoskela} for a proof of this claim. Consequently, we obtain a stronger version of \cite[Proposition 2.17]{heinonenkoskela}, which states that $\Mod_p(E,F)=\Mod_p^c(E,F)$, where $\Mod_p^c(E,F)$ is the modulus computed with only continuous admissible functions. Specifically, our proof shows that $\Mod_p(E,F)=\Mod_p^c(E,F)$ whenever $X,E,F$ satisfy the assumptions of Theorem  \ref{thm:capacity} and $p\in[1,\infty)$. With the more restrictive assumptions that the space $X$ is proper and geodesic, this equality was known \cite[Proposition 7]{keith03}. 
\end{remark}  

The strongest previous result on this problem is due to Keith \cite[Proposition 7]{keith03}. He showed (\ref{eq:cap-equality}) under the assumption that $X$ is proper and geodesic (i.e. that each pair of points $x,y\in X$ can be connected by a rectifiable  curve $\gamma$ with $\len(\gamma) = d(x,y)$). Our result weakens properness to local completeness and separability and removes any geodesic assumption.

The capacity problem gives rise to the definition of a Sobolev space. The Sobolev spaces which we consider are those introduced by Cheeger in \cite{che99}. However, we take the perspective of precise representatives which were studied and introduced by Shanmugalingam in \cite{sha00}. The space of these functions is denoted $N^{1,p}(X)$ with $p\in [1,\infty)$, and consists of all functions $u\in L^p(X)$ that have an upper gradient $g\in L^p(X)$. For $p>1$, the space $N^{1,p}(X)$ is equivalent to variants defined using plans, see \cite{ambgigsav}.  The (semi)norm on this space is denoted $\|\cdot\|_{N^{1,p}(X)}$ which equals the usual Sobolev norm in the case of Euclidean spaces equipped with Lebesgue measure. These notions will be precisely defined below in Section \ref{sec:prelim}.

Our second main result shows the density of continuous functions in the Sobolev space and that all Sobolev functions are quasicontinuous. We say that a function $f:X\to \R\cup\{\infty,-\infty\}$ is {\it quasicontinuous} if for every $\epsilon>0$ there exists an open set $O$ with $\Cp_p(O) < \epsilon$ and so that $f|_{X\setminus O}$ is continuous. Recall that the notion of having zero capacity is a finer notion than having zero measure. Indeed, a set of capacity zero must be of measure zero. However, a set of capacity zero will usually be of smaller Hausdorff dimension. 
\begin{theorem} \label{thm:density-main} Let $(X,d,\mu)$ be a locally complete and separable metric measure space. Then $C(X) \cap N^{1,p}(X)$ is dense in $N^{1,p}(X)$ for $p\in[1,\infty)$, and every function $f\in N^{1,p}(X)$ is quasicontinuous.
\end{theorem}
Note that, if $\Omega$ is a domain in a locally complete space, then $\Omega$ is itself locally complete. Thus, the theorem directly applies to domains.  This strengthens the main result in \cite[Theorem 1.1]{bjornnages} in two ways: first, one does not need to switch representatives of $f$, and second, the assumptions are much weaker. 

A similar conclusion is contained in \cite[Theorem 4.1]{sha00}, under the additional hypothesis that  $X$ is  complete  and measure doubling, while also satisfying a Poincar\'e inequality. On the other hand, by just assuming completeness and separability and measure doubling,  it was shown in \cite{ambgigsav} that Lipschitz functions are dense when $p>1$. In \cite{teriseb}, this result was slightly extended to complete and separable metric spaces with finite Hausdorff dimension, and for all $p\in [1,\infty)$. These three results prove density of \emph{Lipschitz functions}, but with more restrictive assumptions - all of them require the space to have finite Hausdorff dimension. In contrast, our theorem removes any assumption on the dimension of $X$, but a price for this is paid in the weaker conclusion: the density of \emph{continuous} functions. Thus, Theorem \ref{thm:density-main} substantially answers a question from \cite{bjornnages} on whether continuous functions are dense in Sobolev spaces without any further assumptions.

\begin{remark} We note that Theorem \ref{thm:density-main} and Theorem \ref{thm:capacity} are related, but neither is implied directly by the other. In particular, we can not prove Theorem \ref{thm:capacity} directly by approximation, since \emph{a priori} Cauchy sequences of Sobolev functions only converge almost everywhere, and the sets $E,F$ in Theorem \ref{thm:capacity} may have measure zero.

Further, it is worth noting, that it is an interesting open problem to determine if Lipschitz functions are also dense in the Sobolev space for any complete and separable metric space equipped with a Radon measure, which is finite on balls.
\end{remark}

Next, consider the capacity of a set $E\subset X$ defined as
\[
\Cp_p(E) := \inf\{\|u\|_{N^{1,p}(X)}^p : u|_E \geq 1\}.
\]
We first establish a crucial technical result, which allows us to strengthen the results on Sobolev spaces from \cite{bjornnages}. Namely, we prove the outer regularity of $\Cp_p(E)$ in locally complete spaces. 

\begin{theorem}\label{thm:capacityouterreg}
Suppose that $(X,d,\mu)$ is locally complete and separable metric measure space. Let $E\subset X$ be any set. Then

$$\Cp_p(E) = \inf_{E\subset O} \Cp_p(O),$$ 
where the infimum is taken over open subsets of $X$ containing $E$.
\end{theorem}

This improves on prior work by removing the assumption of properness and density used in \cite[Corollary 1.3]{bjornnages}. The proof  involves both Theorem \ref{thm:density-main} and an observation in Proposition \ref{prop:funconstr} on lower semicontinuity involving certain ``good'' functions. (These are used to handle the case  when $\Cp_p(E)=0$; see Proposition \ref{prop:capacity-null-case}.)

As a corollary, we show that  Sobolev Capacity is a Choquet capacity, under very weak assumptions. See Section \ref{def:Choquet} for a definition of a Choquet capacity.

\begin{corollary} \label{cor:capacity-choquet} If $(X,d,\mu)$ is a locally complete and seperable metric measure space and $p\in (1,\infty)$, then the map $E\mapsto \Cp_p(E)$, for $E\subset X$, is a Choquet capacity.
\end{corollary}

\begin{remark}\label{rmk:neighborhoodcapacity} In much of the literature, see e.g. \cite{kinnunenmartio, HKM06}, a neighbourhood capacity is defined: 

\[
\overline{\Cp_p}(E) := \inf\{\|u\|_{N^{1,p}(X)}^p : u|_O \geq 1, \text{ for an open set } O \text{ with } E \subset O\}.
\]
An advantage of this definition is that it is automatically outer regular and a Choquet capacity without further assumptions, see \cite{kinnunenmartio}. Using Theorem \ref{thm:capacityouterreg} it is easy to show that $\Cp_p(E)=\overline{\Cp_p}(E)$ for locally complete and separable metric measure spaces. This gives another way of proving Corollary \ref{cor:capacity-choquet}. 

Much of the literature is split on which definition, $\Cp$ or $\overline{\Cp_p}$, they employ. Theorem \ref{thm:capacityouterreg} shows that very generally the two coincide, and one can use either definition and obtain an equivalent theory.
\end{remark}

In conclusion, we discuss the ways in which we improve on prior work, such as \cite{bjornnages}, and how we execute this technically. First, Theorem \ref{thm:density-main} rests on a new approximation inspired by the authors' prior work in \cite{seb2020,duality}. This approximation is built by solving an extension problem. Let $K\subset X$ be compact such that $f|_K$ is continuous. Theorem \ref{thm:Extension} describes how, and under which assumptions, we are able to extend $f|_K$ to a continuous function $\tilde{f}\in N^{1,p}(X)$. See Equation \eqref{eq:approximation} for the precise formulation of this extension. This construction ought to be thought of as a discretized and adapted version of the more familiar construction used in Proposition \ref{prop:funconstr}, see in particular Proposition \ref{prop:uppergradientpaths}. The discretization is our main new contribution and yields continuity without assuming the existence of curves. It plays a crucial role in allowing us to dispense with the geodesic assumption employed in \cite{keith03}, and the quasiconvexity assumption employed in \cite{heinonenkoskela}.
This approach to approximating Sobolev functions by discretizations is novel. Indeed, prior methods fell  short from being able to handle the case of complete and separable metric spaces.

A second technical contribution of this paper concerns removing the properness assumption in \cite{bjornnages}. This is somewhat subtle, and involves the inner-regularity of the measure $\mu$, i.e. that for any bounded Borel set $A\subset X$ and $\epsilon>0$ there is a compact set $K\subset A$ with $\mu(A\setminus K) < \epsilon$. The main argument here is in Proposition \ref{prop:funconstr}, where a slight modification of the notion of ``good function'' allows for the usual arguments in \cite[Section 3]{bjornnages} to go through. Indeed, this yields Proposition \ref{prop:capacity-null-case} and  the more general capacity results. A refinement of this notion, ``a good sequence of functions'', plays a role in the proof of Theorem \ref{thm:density-main}. 

In Section \ref{sec:prelim} we set the notation and prove some useful lemmas about good functions, discrete paths and almost upper gradients. The results in that section are new and have been written in a way that they may be useful in future work. In Section \ref{sec:complete}, we establish our main results in the complete setting. In Section \ref{sec:locallycomplete}, we extend these results to the locally complete setting using partition of units and localization. Finally, in Section \ref{sec:choquet}, we discuss the Choquet property and the equivalence of different definitions of capacity.

\section{Notation and preliminaries}\label{sec:prelim}

\subsection{Modulus and Sobolev spaces}\label{sec:mod-sob}

Throughout the paper $X$ will be a separable metric space and $\mu$ any Borel measure on $X$ which is finite and positive on each ball, that is $\mu(B(x,r)) \in (0,\infty)$ for each ball $B(x,r)\subset X$. Such measures are Radon when $X$ is (locally) complete and separable, see\cite[Theorem 7.1.7,  Definition 7.1.1]{bogachev07}. (In the reference, the claim is stated only for complete metric spaces. However, by an extension of the measure to the completion, following \cite{saksman}, we obtain the claim for locally complete spaces.). In particular, the measures $\mu$ in this paper are inner and outer regular.

By convention, we denote open  balls by $B(x,r)=\{y\in X: d(x,y) < r\}$. The value $r$ is called the radius of the ball (which may be non-unique), and any ball of radius $r$ is referred to as an \emph{$r$-ball}. The distance between two sets $A,B\subset X$ is defined as $d(A,B)=\inf_{a\in A,b\in B} d(a,b)$. For a single point $x\in X$, we adopt the convention $d(x,A)=d(\{x\},A)$. The characteristic function of a set $A\subset X$ is denoted $\ones_A$.

Generally, we will assume that either $X$ is complete or locally complete. In the latter case, we will also consider its completion $\hat{X}$. If $X$ is locally complete, then $X$ is an open subset in $\hat{X}$. The spaces of $L_p$-integrable functions with respect to $\mu$ for $p\in[1,\infty)$ will be denoted by $L^p(X)$. The $L^p$-norm of a function $f$ is denoted $\|f\|_{L^p(X)}$. The space of continuous functions on $X$ is denoted $C(X)$. We do not need a topology on this space, and thus consider it only as a set.

To discuss Newtonian spaces and capacities we next recall some classical terminology. These are covered in more detail in \cite{shabook}, as well as \cite{sha00, bjornnages}.

A curve $\gamma$ is a continuous map $\gamma:[0,1] \to X$ (or, in specific instances any continuous map $\gamma: I \to X$, where $I=[a,b]\subset \R$ is a bounded interval). The length of a rectifiable curve is denoted $\len(\gamma)$. The {\it speed} of an absolutely continuous curve, which exists for a.e. $t\in [0,1]$, is defined as 
\begin{equation}\label{eq:curve-speed}
|\gamma'(t)|=\lim_{h\to 0}\frac{d(\gamma(t+h),\gamma(t))}{h}.
\end{equation}
 Every rectifiable curve has a unique constant-speed parametrization, where $|\gamma'(t)|=\len(\gamma)$ for a.e. $t\in [0,1]$, \cite[Sec. 5.1]{shabook}. If $\tilde{\gamma}:[0,1]\to X$ is the  constant-speed parametrization of $\gamma$, we define the
path integral with respect to $\ga$ as:
\begin{equation}\label{def:curveintegral}
\int_\gamma g \, ds \defeq \int_{0}^1 g(\tilde{\gamma}(t)) |\tilde{\gamma}'(t)| \, dt = \len(\gamma) \int_{0}^1 g(\tilde{\gamma}(t))  \, dt
\end{equation}
when $g$ is any Borel function for which the right hand-side is defined. 

We will mostly only consider rectifiable curves and, unless otherwise specified, allow constant curves. We write $\gamma \subset A$ for a subset $A\subset X$ if $\gamma([0,1]) \subset A$. If $x\in X$ is any point, we write $\gamma:A\leadsto x$ to denote that $\gamma(0)\in A$ and $\gamma(1)=x$, i.e. that $\gamma$ connects $A$ to $x$. The diameter of a curve is denoted $\diam(\gamma)\defeq \diam(\image(\gamma))=\sup_{s,t\in [0,1]} d(\gamma(s),\gamma(t)).$  

Let $\Gamma$ be a collection of rectifiable curves. A non-negative Borel function $\rho:X\to[0,\infty]$ is called admissible for $\Gamma$, denoted $\rho\in \adm(\Gamma)$, if $\int_\gamma \rho \, ds \geq 1$ for each $\gamma\in \Gamma.$ Here, $\int_\gamma g \, ds$ is the path integral defined in (\ref{def:curveintegral}). Modulus is defined by

\[
\Mod_p(\Gamma) = \inf_{\rho\in \adm(\Gamma)} \|\rho\|_{L^p(X)}^p.
\]
A property is said to hold for $p$-a.e. curve $\gamma$ if it holds for each rectifiable $\gamma\not\in \Gamma$ for some collection $\Gamma$ with $\Mod_p(\Gamma)=0$. Given two sets $E,F\subset X$ we will denote by $\Ga(E,F)$ the family of all rectifiable curves $\ga$ in $X$ with $\ga(0)\in E$ and $\ga(1)\in F$.

Recall, that a non-negative Borel function $g:X\to [0,\infty]$ is called an {\it upper gradient} for $f:X\to[-\infty,\infty]$, if for every rectifiable $\gamma:[0,1]\to X$, we have
\begin{equation}\label{eq:ug2}
|f(\gamma(1))-f(\gamma(0))|\leq\int_\gamma g ~ds.
\end{equation}
Here, the left hand side is interpreted to be infinity if the expression gives $|\infty-\infty|$ or $|-\infty-(-\infty)|$.
The collection of upper gradients for $f$ is denoted by $\mathcal{D}(f)$.

We define the Newtonian space $N^{1,p}(X)$ as the collection of all functions $f\in L^p(X)$ that admit an upper gradient $g\in L^p(X)$. A seminorm on $N^{1,p}(X)$ is given by
\[
\|f\|_{N^{1,p}(X)} = \left(\|f\|_{L^p(X)}^p + \inf_{g\in \mathcal{D}(f)} \|g\|_{L^p(X)}^p\right)^{1/p}.
\]
Then, if we identify $f\sim g$ for $f,g\in N^{1,p}(X)$ whenever $\|f-g\|_{N^{1,p}(X)}=0$, we obtain a Banach space; see \cite{sha00}. Thus, while formally $N^{1,p}(X)$ consists of equivalence classes of functions, we will always consider pointwise representatives  for a given class. 

A function $g$ is a ($p$-){\it weak upper gradient}, if Inequality \eqref{eq:ug} holds for $p$-a.e.  rectifiable curve $\gamma:[0,1]\to X$. A  function $f\in L^p(X)$ always admits a minimal $p$-weak upper gradient $g_f$ for which $\|g_f\|_{L^p(X)}=\inf_{g\in \mathcal{D}(f)} \|g\|_{L^p(X)}$. See \cite[Theorem 6.3.20]{shabook} for further details. The following is a classical statement following from  the Vitali–Carath\'eodory theorem.

\begin{lemma}\label{lem:lowersem} If $g\in L^p(X)$ is any $p$-weak upper gradient for $f$, then for any $\epsilon>0$, there exists a lower semicontinuous $g_\epsilon\geq g$ so that $g_\epsilon$ is an upper gradient for $f$ and $\int_X g_\epsilon^p d\mu \leq \int_X g^p d\mu + \epsilon$. 
\end{lemma}

\begin{proof} Let $g\in L^p(X)$ be a weak upper gradient. If $\Gamma$ is the family of rectifiable curves so that Inequality \eqref{eq:ug} does not hold, then $\Mod_p(\Gamma)=0.$ Hence, by \cite[Lemma 5.2.8]{shabook}, for any $\epsilon>0$, there is a $h$ so that $\int_X h^p d\mu \leq \epsilon/2$ and $\int_\gamma h ds = \infty$ for each $\gamma\in \Gamma$. Applying  \cite[Vitali–Carath\'eodory theorem, p.~108]{shabook} to $\max(g,h)$ we obtain a function $g_\epsilon$ such that $g_\epsilon \geq \max(g,h)$, with
$\int_X g_\epsilon^p d\mu \leq \int_X g^p d\mu+\epsilon$. Finally, we verify that  Inequality \eqref{eq:ug} holds for $g_\epsilon$ and for every rectifiable path $\gamma$. Indeed, if $\gamma\in \Gamma$, then \eqref{eq:ug} follows  from $\infty = \int_\gamma h ds \leq \int_\gamma g_\epsilon ds$. While, for $\gamma \not\in \Gamma$, Inequality \eqref{eq:ug} is satisfied since it holds for $g$ and $g\leq g_\epsilon$.
\end{proof}
If $E\subset X$, denote by $\Gamma_E$ the set of non-constant rectifiable curves that intersect $E$. A set $E$ is called $p$-exceptional if $\Mod_p(\Gamma_E)=0$. 

We will need a version of \cite[Lemma 4.3]{sha00}, see also \cite[Lemma 6.3.14]{shabook} and \cite[Proposition 2.22.]{che99} which we state next.

\begin{lemma}\label{lem:uppergradients} Suppose that $f\in N^{1,p}(X)$ has $g\in L^p(X)$ as upper gradient, and suppose $f|_A = c$ for some $c\in \R$, and for some Borel set $A\subset X$. Then, the function $g_A = g\ones_{X\setminus A}$ is a $p$-weak upper gradient for $f$. In particular, the minimal $p$-weak upper gradient $g_f$ satisfies $g_f(x)=0$ for $\mu$-almost every $x\in A$.
\end{lemma}

Proposition \ref{prop:extension} is useful when extending Sobolev functions. It differs from Lemma \ref{lem:uppergradients} in a crucial way, that we do not need to assume that a given function $\tilde{f}$ is \emph{a priori} a Sobolev function. Our starting point is a Sobolev function $f$ and its upper gradient $g$. Another function $\tilde{f}$ agrees with $f$ on a set $K$, and we \emph{a priori} know that $g$ is also an upper gradient for $\tilde{f}$ in $X\setminus K$. Here, we say that $g$ is an upper gradient for $f$ in a set $A$, if Inequality \eqref{eq:ug} holds for every curve $\gamma$ with $\gamma \subset A$. When $\tilde{f}$ is continuous this information can be patched together to conclude that $g$ is an upper gradient for $\tilde{f}$ in all of $X$. As a consequence, this shows that $\tilde{f}$ is Sobolev. The proposition has a proof which is quite similar in spirit to \cite[Lemma 4.3]{sha00}. However, given the differences in the statements and some details of the arguments, we provide a complete proof.

\begin{proposition}\label{prop:extension} Let $f\in N^{1,p}(X)$  and let $g \in L^p(X)$ be an upper gradient for $f$. Let $\tilde{f}\in L^{p}(X)$ be a continuous function so that that $f|_K=\tilde{f}|_K$ for some closed set $K$. If $g$ is an upper gradient for $\tilde{f}$ in $X\setminus K$, then $g$ is also an upper gradient for $\tilde{f}$ in all of $X$. In particular, $\tilde{f}\in N^{1,p}(X)$.
\end{proposition}

\begin{proof} Let $\gamma:[0,1]\to X$ be any non-constant rectifiable curve. The upper gradient inequality \eqref{eq:ug} is  invariant under reparametrizations. Hence, for convenience, we will assume that $\gamma$ has the constant-speed parametrization. There are essentially three cases to consider, when verifying Inequality \eqref{eq:ug} for $\tilde{f}$ in place of $f$. Since \eqref{eq:ug} is clear when $\int_\gamma g \, ds = \infty$, we can assume that $\int_{\gamma} g \, ds < \infty$. 

\textbf{1. Assume $\gamma(0),\gamma(1)\in K$.} Since $g$ is an upper gradient for $f$ and $f|_K=\tilde{f}|_K$ the Inequality \eqref{eq:ug} for $\gamma$ and $\tilde{f}$ is identical to that for $f$. 

\textbf{2. Assume $\gamma(0)\in K$ but $\gamma(1)\not\in K$ (or the reverse).} The reverse case of $\gamma(1) \in K$ and $\gamma(0) \not\in K$ is symmetrical and can be reduced to this by considering the curve $\tilde{\gamma}(t)=\gamma(1-t)$.  Thus take $\gamma(0)\in K$ and $\gamma(1)\not\in K$.

Let $t=\sup\gamma^{-1}(K)$. We have $t<1$ since $K$ is closed. Consider $\gamma_1=\gamma|_{[0,t]}$ (which may be constant) and $\gamma_{2,\epsilon}=\gamma|_{[t+\epsilon,1]}$ for $\epsilon \in [0,1-t)$. By case 1., we have $\int_{\gamma_1} g ds \geq |\tilde{f}(\gamma_1(0))-\tilde{f}(\gamma_1(t))|.$
 
For $0<\epsilon<1-t$, we have that $\gamma_{2,\epsilon}\subset X\setminus K$ and since $g$ is an upper gradient for $\tilde{f}$ in $X\setminus K$ we have 
\begin{equation}\label{eq:ug-outside}
 |\tilde{f}(\gamma_{2,\epsilon}(t+\epsilon))-\tilde{f}(\gamma_{2,\epsilon}(1))|\le \int_{\gamma_{2,\epsilon}} g \, ds 
\end{equation}
Thus, we get
\begin{align*}
|\tilde{f}(\gamma(0))-\tilde{f}(\gamma(1))|&\leq |\tilde{f}(\gamma(0))-\tilde{f}(\gamma(t))|+|\tilde{f}(\gamma(t))-\tilde{f}(\gamma(1))| \\
& \leq \int_{\gamma_1} g \,ds + \lim_{\epsilon\to 0} |\tilde{f}(\gamma_{2,\epsilon}(t+\epsilon))-\tilde{f}(\gamma_{2,\epsilon}(1))| & \text{(by item {\bf 1.} and continuity)}\\
& \leq \int_{\gamma_1} g \,ds + \lim_{\epsilon\to 0} \int_{\gamma_{2,\epsilon}} g \,ds & \text{(by (\ref{eq:ug-outside}))}\\
& = \int_{\gamma} g \,ds.
\end{align*}  
In the second to last line, we rewrite the integrals using $g(\gamma(t))$ multiplied by the characteristic function of  $[0,t]\cup[t+\epsilon,1]$, and then we conclude using monotone convergence.

\noindent \textbf{3. Assume $\gamma(0),\gamma(1)\not\in K$.} If $\gamma\subset X\setminus K$, then the claim follows since $\tilde{g}$ is an upper gradient for $\tilde{f}$ in $X\setminus K$. Otherwise there is some $t\in [0,1]$ so that $\gamma(t) \in K$. Now, apply the second case to $\gamma|_{[0,t]}$ and to $\gamma|_{[t,1]}$ together with the triangle inequality to get Inequality \eqref{eq:ug}.
\end{proof}

Let $N^{1,p}_b(X) \subset N^{1,p}(X)$ consist of those functions $f\in N^{1,p}(X)$ with bounded support which are bounded in $X$. More precisely, $N^{1,p}_b(X)$ consists of those $f\in N^{1,p}(X)$ for which there are constants $M,R>0$ and a point $x_0 \in X$, so that $f|_{X\setminus B(x_0,R)}=0$ almost everywhere and $f(x) \in [-M,M]$ for almost every $x\in X$. An important first step will be to reduce the approximation to such functions. This result is very standard, and can be found in many references, see e.g. \cite[Lemma 2.14]{shaharmonic}. We provide a proof for the sake of completeness.

\begin{lemma}\label{lem:densityN1pb} $N^{1,p}_b(X)$ is dense in $N^{1,p}(X)$.
\end{lemma}
\begin{proof} Fix $x_0 \in X$ and consider $M>0$. Let $\psi_M(x) = \max(\min(2-d(x_0,x)/M,1),0)$, which can be seen to be $1/M$-Lipschitz.  Define

$$f_M = \psi_M(x) \min(\max(f,-M),M).$$

We have $|f_M|\leq \min(|f|,M)$, and $\lim_{M\to\infty}f_M = f$ pointwise and in $L^p(X)$. Further, using the Leibniz rule for Sobolev functions (see \cite[Proposition 6.3.28]{shabook}) one can show that $g_{f_M} \leq \frac{|f|}{M} + g_f$ is an upper gradient for $f_M$. So, $f_M \in N^{1,p}(X)$. Also $f_M - f = 0$ on the set $A_M=B(x_0,M) \cap \{|f|\leq M\}$. By Lemma \ref{lem:uppergradients}, the function $f_M-f$ has a weak upper gradient $g_{f_M-f} \leq \ones_{X \setminus A_M} (2g_f + \frac{|f|}{R})$. So $g_{f_M-f}\to 0$ in $L^p$ by dominated convergence, since $\mu(X\setminus \bigcup_{M\in\N}A_M)=0$. Thus

$$\lim_{M \to \infty}\|f_M-f\|_{N^{1,p}(X)}^p = \lim_{M\to\infty} \|f-f_M\|_{L^p(X)}^p+\|g_{f-f_M}\|_{L^p(X)}^p=0.$$
\end{proof}

\subsection{Good functions}

We will mostly consider curve families $\Gamma$ which are invariant under re-parametrization. That is, if $\gamma\in \Gamma$, then any reparametrization of the curve is in $\Gamma$ as well. With this in mind, we say that a collection $\Gamma$ of curves is pre-compact, if every sequence $\{\gamma_i\}_{i=1}^\infty\subset \Gamma$,  where $\gamma_i:[0,1]\to X$ is parametrized by constant-speed, has a uniformly convergent subsequence. Many arguments regarding modulus rely on extracting convergent subsequences. The basic requirement is some form of compactness, and the following formulation of Arzel\`a-Ascoli captures this.

\begin{lemma}(Arzel\`a-Ascoli)\label{lem:arzela} Suppose that $X$ is complete and that $L\geq 1$. A collection $\Gamma$ of curves of length at most $L$ is pre-compact, if given the constant-speed parametrizations $\gamma:[0,1]\to X,$ the set 
\[
A_t=\{\gamma(t): \gamma\in \Gamma, \gamma \text{ parametrized by constant speed}\}
\]
is pre-compact in $X$, for every $t\in [0,1]$.
\end{lemma}

The proof is standard, see for instance the argument in \cite[Theorem 4.25]{walterreal}. When $X$ is a proper space, pre-compactness is the same as a boundedness. However, to work in the case when $X$ is simply a  complete space, we introduce a notion of ``good function'', which allows us to circumvent the lack of properness. 

Given a collection $\Gamma$ of curves and a set $A$, the collection of curves $\Gamma^A$ contained in $A$ is defined by $\Gamma^A=\{\gamma \in \Gamma: \gamma \subset A\}$. If further $\delta, L>0$ and $g\in L^p(X)$, we define a subcollection by 
\begin{equation}\label{eq:ga-adel}
\Gamma_{\delta,L}^A(g)=\{\gamma \in \Gamma^A: \int_\gamma g \leq L, \diam(\gamma) \geq \delta\}.
\end{equation}

\begin{definition}\label{def:goodfunction} A Borel function $g:X\to [0,\infty]$ is called a 
\emph{good function}, if the set of curves $\Gamma_{\delta,L}^A(g)$ defined in (\ref{eq:ga-adel})
is pre-compact, for any family $\Gamma$ of rectifiable curves, any bounded Borel set $A$, and for all $\de, L>0$.
\end{definition}

\begin{example} Suppose that $X$ is compact. For any $\epsilon>0$ the function $g:X\to [\epsilon,\infty]$ is a good function. Indeed, let $\Gamma$ be an arbitrary family of rectifiable curves and let $A\subset X$ be a bounded Borel set. In this case, for any $\delta, L >0$ we have $\Gamma^A_{\delta,L}(g) \subset \{\gamma \subset A : \len(\gamma) \leq L/\epsilon \}$. Then, $A_t \subset X$, and $A_t$ is automatically pre-compact as a subset of a compact space $X. $ Therefore, by Lemma \ref{lem:arzela} the collection $\Gamma^A_{\delta,L}(g)$ is pre-compact.
\end{example}

The next lemma strenghtens the Vitali-Carath\'eodory Lemma \ref{lem:lowersem} by showing that any $L^p$-function can be slightly modified  to become a lower semicontinuous good function.

\begin{lemma}\label{lem:goodfunc} Assume $X$ is complete. If $g\in L^p(X)$ is non-negative, then for any $\epsilon>0$ there exists a lower semicontinuous good function $g_\epsilon \geq g$ so that $\|g_\epsilon\|_{L^p(X)}\leq \|g\|_{L^p(X)} + \epsilon$.
\end{lemma}

\begin{proof}
Fix $g,\epsilon$ as in the statement. By an application of Lemma \ref{lem:lowersem}, we can take $g$ to be lower semicontinuous. Fix any point $x_0 \in X$.

Choose compact sets $\tilde{K}_i \subset B(x_0,i)$, so that $\mu(B(x_0,i)\setminus \tilde{K}_i) \leq \epsilon^p 2^{-(i+1)p}$. This is possible since $\mu(B(x_0,i))<\infty$ and since the measure is Radon. Define $K_i = \bigcup_{j=1}^i \tilde{K}_j$, so that $K_n\subset K_m$ for $n\leq m$ and so that $\mu(B(x_0,i)\setminus K_i) \leq \epsilon^p 2^{-(i+1)p}$.
Define 
\[
g_{\epsilon} = g + \sum_{i=1}^\infty \left(\ones_{B(x_0,i) \setminus K_i} + \frac{\epsilon}{2^{i+1}\mu(B(x_0,i))^{1/p}} \ones_{B(x_0,i)}\right).
\]
This function is  lower semi continuous, since it is a sum of lower semicontinuous functions.
Moreover, 
\[
\|g_{\epsilon}\|_{L^p(X)} \leq \|g\|_{L^p(X)} + \sum_{i=1}^\infty \left(\left\|\ones_{B(x_0,i)\setminus K_i}\right\|_{L^p(X)} + \left\|\frac{\epsilon}{2^{i+1}\mu(B(x_0,i))^\frac{1}{p}} \ones_{B(x_0,i)}\right\|_{L^p(X)}\right) \leq \|g\|_{L^p(X)} +\epsilon.
\]

We show that $g_\epsilon$ is a good function. Let $\Gamma$ be any family of rectifiable curves and let $A$ be any bounded Borel set. There exists some $j\in \N$ so that $A\subset B(x_0,j)$. By increasing the set $A$, we get a larger collection of curves, and thus it suffices to consider $A= B(x_0,j)$. The idea is to use Arzela\`-Ascoli Lemma  \ref{lem:arzela} to show that the collection $\Gamma^{A}_{\delta, L}(g_\epsilon)$ defined in (\ref{eq:ga-adel}), for given $\delta,L>0$, is pre-compact.  
First note that, if $\gamma \in \Gamma^A_{\delta, L}(g_\epsilon)$, then $L \geq \int_\gamma g_\epsilon ds$. Since $A = B(x_0,j)$, we get $g_\epsilon>\frac{\epsilon}{2^{j+1}\mu(B(x_0,j))^\frac{1}{p}} \ones_A$. Thus, since $\ga\subset A$,  
\[
\len(\gamma) \leq \frac{2^{j+1}\mu(B(x_0,j))^\frac{1}{p}}{\epsilon} \int_\ga g_{\epsilon} \,ds \leq \frac{2^{j+1}\mu(B(x_0,j))^\frac{1}{p}L}{\epsilon}=: L'.
\]
Therefore all the curves in $\Gamma^A_{\delta, L}(g_\epsilon)$ have length at most $L'$.

Now, consider curves $\gamma\in \Gamma_{\delta,L}^A(g_\epsilon)$ which are parametrized by constant-speed. It is enough to show that for each $t\in [0,1]$ the set 
\[
A_t:=\{\gamma(t) : \gamma\in \Gamma_{\delta,L}^A(g_\epsilon) \text{ is parametrized by constant speed}\}
\]
 is pre-compact in $X$. 
Since $X$ is complete, it suffices to show that $A_t$ is totally bounded. Fix $\eta\in(0,\delta/2)$ for this purpose. Choose $N = j+\lfloor 4L/\eta\rfloor+1$. We claim that $A_t\subset \{y\in X: d(y,K_N) < \eta/4\}$. Assume, for the sake of contradiction, that
$d(\gamma(t),K_N)\geq \eta/4$ for some $\gamma\in \Gamma_{\delta,L}^A(g_\epsilon)$. Since $\diam(\gamma) \geq \de > 2\eta$, we must have a segment of $\gamma$ of length at least $\eta/4$ contained in $A\setminus K_N \subset B(x_0,j) \setminus K_N$. Thus, since $K_i \subset K_N$ for each $i$ with $1\leq i \leq N$, and $A \subset B(x_0,i)$ for each $j\leq i \leq N$, we get
\[
\int_\gamma g_\epsilon \,ds \geq \sum_{i=j}^N \int_\gamma \ones_{B(x_0,i)\setminus K_i}\,ds \geq (N-j) \int_\gamma \ones_{B(x_0,j)\setminus K_N}\,ds \geq \left(\lfloor 4L/\eta\rfloor+1\right)\frac{\eta}{4} > L,
\]
which is a contradiction. 
Thus, $A_t\subset \{y\in X: d(y,K_N) < \eta/4\}$. By covering $K_N$ by $\eta/2$ balls and inflating these balls by $2$ we get a finite covering of $A_t$ by $\eta$-balls. This shows that the set $A_t$ is totally bounded and thus pre-compact. 
\end{proof}

We will need the following lower semicontinuity of curve integrals. The result is classical, and its proof has appeared in many places, such as \cite[Proposition 4]{keith03}. For the reader's convenience we give a short proof.

\begin{lemma}\label{lem:lsc} Let $\gamma_i: [0,1]\to X$ be a sequence of rectifiable curves parametrized by constant speed, with $\len(\gamma_i) \leq L$ for some $L\in [0,\infty]$ and all $i\in \N$, and suppose that $\gamma_i$ converges uniformly to a curve $\gamma:[0,1]\to X$. If $g:X\to [0,\infty]$ is a lower semicontinuous function, then 
\[
\int_\gamma g \, ds \leq \liminf_{i\to\infty} \int_{\gamma_i} g \, ds.
\]
\end{lemma}

\begin{proof} By passing to a subsequence, we can assume that $\lim_{i\to\infty} \int_{\gamma_i} g \, ds$ exists, and that $\len(\gamma_i)\to L'$ for some $L'$. Since $\gamma_i$ are parametrized by constant speed, see (\ref{eq:curve-speed}),  $|\gamma_i'|(t) = \len(\gamma_i)=: L_i$ for each $i\in \N$ and every $t\in [0,1]$, and $\gamma_i$ is $L_i$-Lipschitz. Therefore,  $\gamma$ is $L$-Lipschitz, and $|\gamma'|(t) \leq L$ for a.e. $t\in [0,1]$. 
\begin{align*}
\int_\gamma g \, ds = \int_0^1 g(\gamma(t))|\gamma'|(t) dt &\leq \int_0^1 g(\gamma(t)) L dt \\
&\leq \lim_{i\to\infty} \int_0^1 g(\gamma_i(t)) L_i dt & \text{(by l.s.c. and Fatou's Lemma)}\\
&= \lim_{i\to\infty} \int_{\gamma_i} g \, ds.
\end{align*}
\end{proof}
The next result is a a generalization of a well-known method to construct a function with a given upper gradient \cite[Lemma 7.2.13]{shabook}. However, the proof there that $u$ is lower-semicontinuous requires that $X$ is proper. In our generality of locally complete and complete spaces, this does not suffice. In this generality the crucial idea is to use the notion of a good function. The proof changes slightly as a consequence of this, and we provide full detail for the sake of completeness.

\begin{proposition}\label{prop:funconstr} Let $X$ be complete. Assume that $V\subset X$ is a bounded open set and that $g:X\to [0,\infty]$ is a good function. Then, the function $u:X\to \R$ given by 
\[
u(x) := \min\left(\inf_{\gamma:X\setminus V \leadsto x} \int_\gamma g \,ds,1\right)
\]
is in $L^p(X)$ for $1\le p\le\infty$, and 
has  $g$ as upper gradient. Moreover, $u$ is lower semicontinuous.
\end{proposition}

\begin{remark}\label{rmk:infimum} We adopt the usual conventions for the infimum. If $x\in X\setminus V$, then the constant curve is allowed, hence $u(x)=0$. If there are no curves $\gamma: X\setminus V \mapsto x$, then $u(x)=1$ since the infimum over an emptyset is $\infty$.
\end{remark}

\begin{proof}[Proof of Proposition \ref{prop:funconstr}] If we show that $u$ is lower semicontinuous, then measurability will follow and we will have $u\in L^p(X)$, because, by Remark \ref{rmk:infimum}, $u\leq \ones_V$. Proving that $g$ is an upper gradient for $u$ is classical argument, see \cite[Lemma 3.1]{bjornnages}. We recall this argument now. If $\gamma:[0,1]\to X$ and $\gamma_0:X\setminus V \leadsto \gamma(0)$ are any rectifiable curves, then we can form a curve $\gamma_1: X \setminus V \leadsto \gamma(1)$ by concatenating them. 
Thus, $\int_{\gamma_1} g \, ds = \int_{\gamma_0} g \, ds + \int_{\gamma} g \, ds$. By the definition of $u$, we get $u(\gamma(1)) \leq \int_{\gamma_1} g \, ds \leq \int_{\gamma_0} g \, ds + \int_{\gamma} g \, ds$. Infimizing over $\gamma_0$ yields $u(\gamma(1)) \leq \inf_{\gamma_0} \int_{\gamma_0} g \, ds + \int_{\gamma} g \, ds.$ Since $u(\gamma(1)) \leq 1$, we have $u(\gamma(1)) \leq u(\gamma(0)) + \int_{\gamma} g \, ds$. By reversing the curve, we obtain $u(\gamma(0))  \leq u(\gamma(1)) + \int_{\gamma} g\, ds$. From these two inequalities, we get inequality \eqref{eq:ug}.

Thus the more important part of the proof is to show that $u$ is  lower semicontinuous.  Arguing by contradiction, assume that we can find some sequence $x_i$ converging to $x$ in $X$, with the property $\lim_{i\to \infty} u(x_i) < u(x)-\Delta$ for some $\Delta>0$. 
Since $u$ is non-negative, we must have that $u(x)>\Delta$, hence $x\in V$. Since $V$ is open, there is some ball $B(x,2\delta) \subset V$, with $\de>0$. By convergence, we can throw away finitely many terms and assume that $x_i \in B(x,\delta)$ for all $i\in\N$. We can also pass to a subsequence so that $u(x_i)<u(x)-\Delta$ for every $i\in \N$. 
In particular, we have $u(x_i)<1$. Hence, we can choose curves $\gamma_i : X\setminus V \leadsto x_i$ with
\begin{equation}\label{eq:lessthanone}
\int_{\gamma_i} g ds \leq u(x)-\Delta \leq 1.
\end{equation}
Assume $\gamma_i: [0,1]\rightarrow X$ has the constant-speed parametrization and let 
\[
t_i:=\sup\left\{t \in [0,1] : d(\gamma_i(t),X\setminus V) \leq \frac{\delta}{2i}\right\},
\]
which is the last time $\gamma_i(t)$ is $\frac{\delta}{2i}$ away from $X\setminus V$.

Let $\tilde{\gamma_i}$ be the constant speed parametrization of $\gamma_i|_{[t_i,1]}$ on $[0,1]$, so that $\tilde{\gamma_i} \subset V$, for each $i\in\N$. By \eqref{eq:lessthanone} and since $\tilde{\gamma_i}$ is a subcurve of $\gamma_i$, we have $\int_{\tilde{\gamma_i}} g , ds \leq 1$.

Since $d(x_i, X \setminus V) \geq \delta$, we have $\diam(\tilde{\gamma_i}) \geq \delta/2$. Thus $\tilde{\gamma_i} \in \Gamma^V_{\delta/2, 1}(g)$, where $\Gamma$ is the collection of all rectifiable curves. Since $g$ is a good function, and since $\tilde{\gamma_i}$ are parametrized by constant speed, there is a subsequence $\{\tilde{\gamma}_{i_k}\}_{k\in\N}$ converging uniformly to some continuous function $\gamma:[0,1]\to X$. 

Next, by Lemma \ref{lem:lsc}, we have 
\[
\int_\gamma g ds \leq \liminf_{k\to\infty} \int_{\tilde{\gamma}_{i_k}} g ds \leq u(x)-\Delta.
\]
By construction, $d(\tilde{\gamma}_{i_k}(0), X \setminus V) \le \frac{\delta}{2i_k}$, and $\tilde{\gamma}_{i_k}(1)=x$. Also, $\lim_{k\rightarrow\infty}\tilde{\gamma}_{i_k}(0)=\ga(0)$ and $\lim_{k\rightarrow\infty}\tilde{\gamma}_{i_k}(1)=\ga(1)$. Therefore, sending $k\to\infty$, we get $\gamma(0)  \in X \setminus V$ and $\gamma(1)=x$, namely, $\gamma: X \setminus V \leadsto x$. This leads to
$$u(x) \le \int_\gamma g ds \leq u(x)-\Delta,$$
which is a contradiction.
\end{proof}

At this juncture, we note, that one of the crucial insights of the present paper is that we can replace path integrals with discrete Riemann-type sums over paths. We will next move to develop the language for such arguments. We will also prove an analogue of Proposition \ref{prop:funconstr} in Proposition \ref{prop:uppergradientpaths} for such constructions.

\subsection{Discrete paths}
We will be considering discrete path approximations to curves. A \emph{(discrete) path} is a sequence $P=(p_0,\dots,p_n)$, with $n\geq 0$, and which does not repeat.  We identify $P$ sometimes with the image set, for example in writing $p\in P$ to state that a point $p$ lies in the path. When $n\geq 1$ define its \emph{mesh size} by $\mesh(P) = \max_{k=0,\dots n-1} d(p_k,p_{k+1})$ and its \emph{length} by $\len(P) = \sum_{k=0}^{n-1} d(p_k,p_{k+1})$. The \emph{diameter}, $\diam(P)$, is the diameter as a set of points. Note that the path $P=(p_0)$ consisting of only one point, and with diameter, length and mesh equal to zero, is permitted. Given a function $g:X\to \R$ we define the \emph{discrete integral} of $g$ along $P$ by
\[
\int_P g := \sum_{k=0}^{n-1} g(p_k)d(p_k,p_{k+1}).
\]
Again, if $P$ is a single point, then $\int_P g = 0$.

Discrete paths can be extended to curves after we pass to a larger super space. For such arguments, we introduce an isometric Kuratowski embedding $\iota: X\to \ell_\infty(\N)$ into the sequence space $\ell_\infty(\N)$.  We will fix such an embedding for the remainder of this subsection.  

Given a discrete path $P$ we call a curve $\gamma:[0,1]\to \ell_\infty(\N)$ its \emph{linearly interpolating curve} if it is constructed as follows. If $\len(P)=0$, i.e. when discrete path consists of only single point $P=(p_0)$, then define $\gamma(t)=p_0$ for all $t \in [0,1]$. Next, we define the interpolant when $\len(P)>0$. Let $t_0=0$ and define 
\begin{equation}\label{eq:time-partition}
t_l := \sum_{k=0}^{l-1} \frac{d(p_k,p_{k+1}) }{ \len(P)},
\end{equation}
for $l=1,\dots, n$. We refer to $t_0,\dots, t_n$ as the time-partition points associated to $P$, which are only defined if $\len(P)>0$. Define $\gamma(t_l) = p_l$ and  define 
\begin{equation}\label{eq:curve-interpolation}
\gamma(t) = \frac{t_{l+1}-t}{t_{l+1}-t_l} p_{l} + \frac{t-t_l}{t_{l+1}-t_l} p_{l + 1},
\end{equation}
for $t \in (t_l,t_{l+1})$.  Note that since the path is simple, $p_l\neq p_{l+1}$ and $t_l\neq t_{l+1}$ for all $l=0,\dots, n-1$.

With this construction, we have the following lemma.

\begin{lemma}\label{lem:interpolation} Let $P$ be a discrete path and $\gamma$ its linearly interpolating curve. We have the following.
\begin{enumerate}
\item $\len(\gamma)=\len(P)$.
\item $d(\gamma(t), X) \leq \mesh(P)$ for every $t\in [0,1]$.
\item $\gamma$ is parametrized by constant speed.
\end{enumerate}
\end{lemma}
\begin{proof}
Since $\gamma$ is piecewise linear, we can compute its length by adding up the linear segments, and $\len(\gamma)=\len(P)$ follows. 

If $\len(P)=0$, then $\gamma$ is a constant curve and $d(\gamma,X)=0=\mesh(P)$. Further, $\gamma$ has constant (zero) speed.  
Suppose then that $\len(P)>0$ and let $t_l$, for $l=0,\dots, n-1$, be the time-partition points. For every $t\in [0,1]$,  there is a $l=0,\dots, n-1$ for which $t\in [t_l,t_{l+1}]$. Hence, by (\ref{eq:curve-interpolation}),  
\[
d(\gamma(t),X) \leq \min(d(\gamma(t),p_l), d(\gamma(t),p_{l+1})) \leq d(p_l,p_{l+1}) \leq \mesh(P),
\] 
for each $t\in [0,1]$. 

Next, if $t\in (t_l,t_{l+1})$ for some $l=0,\dots, n-1$, then $\gamma$ is linear in a neighborhood of $t$, and has speed $d(p_l,p_{l+1})/(t_{l+1}-t_l)=\len(P)$, by (\ref{eq:time-partition}) and (\ref{eq:curve-interpolation}). Therefore, $\gamma$ has speed $\len(P)$ at all points $t\not\in\{t_0,\dots, t_n\}$, in other words, it is parametrized by constant speed.
\end{proof}

If $P_i=(p^i_0,\dots p^i_{n(i)})$ is a sequence of discrete paths, we say that it converges to a curve $\gamma:[0,1]\to X$, if $\lim_{i\to\infty}\mesh(P_i)=0$ and the linear interpolation curves $\gamma_i$ converge to $\gamma$ uniformly in $\ell_\infty(\N)$, in the sense that:
\begin{equation}\label{eq:conv-discrete-cont}
\lim_{i\rightarrow\infty} \sup_{t\in[0,1]} \|\gamma_i(t)-\gamma(t)\|_\infty =0.
\end{equation}
Moreover, one can show that this notion of convergence does not depend on the embedding to $\ell_\infty(\N)$.

We will need the following variant of Lemma \ref{lem:lsc}, in the context of discrete paths.

\begin{lemma}\label{lem:lsccurve} Assume that $P_i=(p^i_0,\dots p^i_{n(i)})$ is a sequence of discrete paths converging to a rectifiable curve $\gamma:[0,1]\to X$,
in the sense of (\ref{eq:conv-discrete-cont}). Also, assume that $\liminf_{i\to\infty} \len(P_i)<\infty$. Then, for any lower semicontinuous non-negative function $g:X\to [0,\infty]$ we have

\[
\int_\gamma g \, ds \leq \liminf_{i\to\infty} \int_{P_i} g.
\] 
\end{lemma} 
\begin{proof} If $\len(\ga)=0$, then the inequality is trivial. So we may assume 
 that $\len(\gamma)>0$, which, by (\ref{eq:conv-discrete-cont}),  implies that $\len(P_i)>0$ for all but finitely many $i\in\N$. For convenience, we pass to a subsequence so that $\len(P_i)>0$ for all $i\in\N$.

We first show that if the Lemma has been proven for $g$ continuous, then it follows for lower semicontinuous $g$. Indeed, we can find an increasing sequence of functions $\{g_l\}_{l\in\N}$ of non-negative  continuous functions converging pointwise to $g$ (see \cite[Corollary 4.2.3]{shabook}). For each $l$ we have

\[
\int_\gamma g_l \, ds \leq \liminf_{i\to\infty} \int_{P_i} g_l \leq \liminf_{i\to\infty}  \int_{P_i} g.
\]
Send $l\to\infty$ and use monotone convergence then to conclude the lemma for all $g$ lower semicontinuous.  

Hence, assume that $g$ is continuous. Denote the interpolating paths for $P_i$ by $\gamma_i$, and use superscripts of the form $t_l^i$ when defining $\ga_i$ as in  (\ref{eq:time-partition})  and (\ref{eq:curve-interpolation}).

Fix $\epsilon>0$. Extend $g$ to be continuous on $\ell_\infty(\N)$ using the Tietze extension theorem; see for example \cite{milman97}. Since the image of $\gamma$ is compact, and $g$ continuous, we can find for any $\epsilon>0$ a $\delta>0$ so that if $x,y\in \ell_\infty(\N)$ and 
$ \max \left(d(x,\gamma),d(y,\gamma),d(x,y) \right) < \delta $
then $|g(x)-g(y)|<\epsilon$. Choose an $N$ so that for $i\geq N$, we have $\mesh(P_i) < \delta/2$, $d(\gamma_i(t),\gamma(t)) < \delta/2$ for all $t\in [0,1]$. 

Next, let $i\geq N$ be arbitrary. For every $t\in [t_l^i,t_{l+1}^i]$, we have $d(\gamma_i(t),   \gamma(t_l^i)) \leq \mesh(P_i) < \delta$ and $\gamma(t_l^i)\in P_i \in X$. Thus, by the choice of $\delta$, we get $g(\gamma_i(t))\leq g(\gamma_i(t_l^i)) + \epsilon$. Integrating this inequality and using (\ref{eq:time-partition}) , we get  
\[
\int_{t_l^i}^{t_{l+1}^i} g(\gamma_i(t)) \len(P_i) dt \leq \len(P_i) (t_{l+1}^i-t_l^i) (g(\gamma(t_l^i)) + \epsilon) \leq d(p_l^i,p_{l+1}^i) (g(\gamma(t_l^i)) + \epsilon).
\]
Finally, by  summing over $l=0, \dots n(i)-1$, Lemma \ref{lem:interpolation} gives
\[
\int_{\gamma_i} g ds = \int_0^1 g(\gamma_i(t)) \len(\gamma_i) dt \leq \sum_{l=0}^{n(i)-1} d(p_l^i,p_{l+1}^i) (g(\gamma(t_l^i)) + \epsilon) \leq \int_{P_i} g + \epsilon \len(P_i).
\]
By taking a limit inferior with $i\to \infty$, we get
\begin{equation}\label{eq:gamiPi}
\liminf_{i\to\infty} \int_{\gamma_i} g \leq \liminf_{i\to\infty} \int_{P_i} g +  \epsilon \liminf_{i\to\infty} \len(P_i) \, .
\end{equation}

Next, by Lemma \ref{lem:interpolation}, each $\gamma_i$ is parametrized by constant speed $\len(P_i)$. Thus, the $\gamma_i$ are $\len(P_i)$-Lipschitz. Let $L = \liminf_{i\to\infty} \len(P_i)$. Then, from uniform convergence, we get that $\gamma$ is $L$-Lipschitz. Together with the fact that the functions $h_i(t)=g(\gamma_i(t))$ converge uniformly to the function $g(\gamma(t))$, since $g$ is continuous, we get

\begin{align*}
\int_\gamma g \, ds &\leq \int_0^1 g(\gamma(t)) L dt \\
&= \lim_{i\to\infty} \int_0^1 g(\gamma_i(t)) L dt\\
&\leq \liminf_{i\to\infty} \int_0^1 g(\gamma_i(t)) \len(P_i) dt\\
&= \liminf_{i\to\infty} \int_{\gamma_i} g ds \leq \liminf_{i\to\infty} \int_{P_i} g +  \epsilon L\, .
\end{align*}
Since $\epsilon>0$ was arbitrary, the claim follows.
\end{proof}

We will need the following compactness statement for discrete paths.

\begin{lemma}\label{lem:compactness} If $\{P_i\}_{i\in \N}$ is a sequence of paths in a complete metric space $X$ satisfying
\begin{enumerate}
\item $\lim_{i\to\infty} \mesh(P_i)=0$; 
\item $\len(P_i) \leq S$ for some $S\in(0,\infty)$ and all $i\in \N$; and 
\item for any $\tau>0$ there is a compact set $K_\tau\subset X$ that $\max_{p\in P_i}d(p,K_\tau) \leq \tau$ for all $i\in \N$,
\end{enumerate}
 then a subsequence of $P_i$ converges to a curve $\gamma:[0,1]\to X$ in the sense of (\ref{eq:conv-discrete-cont}).
\end{lemma}

\begin{proof}For each $i\in \N$ let $\gamma_i:[0,1]\to \ell_\infty(\N)$ be the curve linearly interpolating $P_i$. Lemma  \ref{lem:interpolation} states that we have $\len(\gamma_i) \leq S$ and that the curves $\gamma_i$ are parametrized by constant speed. First, we show that a subsequence of $(\gamma_i)_{i\in \N}$ converges uniformly to some curve $\gamma:[0,1]\to \ell_\infty(\N)$.

Fix $t\in [0,1]$. Let $A_t=\{\gamma_i(t): i\in\N\}$. The claim follows from Lemma \ref{lem:arzela}, if we show that $A_t$ is pre-compact. Since $\ell_\infty(\N)$ is complete, it suffices to show that $A_t$ is totally bounded. Fix $\tau>0$ and choose $N\in\N$ so that $\mesh(P_i) \leq \tau/8$ for all $i\geq N$ and a compact set $K_{\tau/8}$ as in the statement. Then, for $i \geq N$, we have 
\[
d(\gamma_i(t),K_{\tau/8})\leq \mesh(P_i) + \max_{p\in P_i}d(p,K_{\tau/8}) \leq \tau/4.
\] 
Set $K' = K_{\tau/8} \cup \bigcup_{j=1}^N \gamma_j$ which is compact. Since $K'$ is compact, it can be covered by a finite collection $\mathcal{B}$ of balls of radius $\tau/2$. Every point $\gamma_i(t)\in A_t$ has $d(\gamma_i(t),K') \leq \tau/4$, and thus by inflating each ball in $\mathcal{B}$ by two we can cover $A_t$ by finitely many balls of radius $\tau$. Therefore, $A_t$ is totally bounded and pre-compact as desired. 

Thus, a subsequence $\gamma_{i_k}$ converges uniformly to some curve $\gamma:[0,1]\to \ell_\infty(\N)$. Further, for any $t\in [0,1]$ we have $d(\gamma(t),X)=\lim_{k\to \infty} d(\gamma_{i_k}(t), X)\leq \lim_{k\to\infty} \mesh(P_i)=0$ by Lemma \ref{lem:interpolation}. Thus, the image of $\gamma$ is contained in $X$ and the claim follows.
\end{proof}

Discrete paths can be used to conveniently define functions which have given upper gradients, in the spirit of Proposition \ref{prop:funconstr}.

\begin{proposition}\label{prop:uppergradientpaths} Suppose $X$ is a metric space. Let $\delta,M>0$ and let $E\subset X$ be a non-empty subset. Let $g:X\to [0,\infty)$ be a continuous function and $h:E\to \R$ a bounded function. Define
\[
f(y):=\min\left(\inf_{P} \left\{h(p_0)+\int_P g\right\}, M\right),
\]
where the infimum is taken over all discrete paths $P=(p_0,\dots, p_n)$ with $p_0\in E, p_n=y$ and $\mesh(P)\leq \delta$.

Then, $f$ is locally Lipschitz, and $g$ is an upper gradient of $f$.
Moreover, if $h$ is constant and less than or equal to $M$, then $f\equiv h$ on $E$.
\end{proposition}
\begin{proof}
Fix $y\in X$ and assume that $d(x,y)\le \de$. Let $P=(p_0,\ldots,p_n)$ be a discrete path with $p_0\in E, p_n=y$ and $\mesh(P)\leq \delta$. Let $p_{n+1}=x$ and set $P'=(p_0,\ldots,p_k)$, where $k=\inf\{j\ge 0: p_j=x\}$. Thus, $P'$ is a discrete path with $p_0\in E$, $p_k=x$, and $\mesh(P')\le\de$. In particular,
\[
f(x)\le h(p_0)+\int_{P'}g\le h(p_0)+\int_{P}g+g(y)d(x,y).
\]
Taking the infimum over $P$ and comparing with $M$ we get
\begin{equation}\label{eq:g-fxfy}
f(x)\le f(y) + g(y)d(x,y)
\end{equation}
By switching the role of $x$ and $y$, we find that
\begin{equation}\label{eq:maxg-lip}
|f(y)-f(x)|\leq \max(g(x),g(y))d(x,y),
\end{equation}
whenever $d(x,y)\le \de$. Since $g$ is continuous, it is also locally bounded. Hence, (\ref{eq:maxg-lip}) implies that $f$ is locally Lipschitz.

Next, we want to show that $g$ is an upper gradient for $f$. Let $\ga:[0,1]\rightarrow X$ be a curve with constant speed and length $L$. Fix a partition $s_0=0<s_1<\cdots<s_k=1$. Then,
\begin{align*}
\sum_{j=1}^k g(\ga(s_{j-1}))L|s_j-s_{j-1}| & =   \sum_{j=1}^k g(\ga(s_{j-1}))\length\left(\ga|_{[s_j,s_{j-1}]}\right) \\
& \ge  \sum_{j=1}^k g(\ga(s_{j-1}))d\left(\ga(s_{j-1}),\ga(s_j)\right) \\
& \ge  \sum_{j=1}^k \left(f(\ga(s_j))-f(\ga(s_{j-1}))\right) & \text{( by (\ref{eq:g-fxfy}))} \\
& = f(\ga(s_k))-f(\ga(s_0)).
\end{align*}
Taking the limit as the mesh goes to zero, we find that
\[
 f(\ga(1))-f(\ga(0)) \le \int_\ga g ds.
\]
Running $\ga$ in reverse, substituting $s$ with $1-s$, we find by a similar argument that
\[
|f(\ga(1))-f(\ga(0))| \le \int_\ga g ds.
\]
This shows that $g$ is an upper gradient for $f$.

Finally, assume that $h$ is constant and less than or equal to $M$. Then, for $y\in E$, we may choose the constant discrete path $P=(y)$, and get that $f(y)\le h(y)=h$. Conversely, for any non-constant discrete path $P=(p_0,\dots,p_n)$ with $p_0\in E$, $p_n=y$, and $\mesh(P)\le \de$, we have
\[
h(p_0)+\int_P g\ge h(p_0)=h.
\]
Taking the infimum over such paths, and then the minimum with $M$, we find that $f(y)\ge h$.
Therefore, $f(y)=h$.
\end{proof}

\subsection{Good sequences of functions}

For sequences of discrete paths it becomes often convenient to construct a ``good sequence of functions'', which approximate a given function.

\begin{definition}\label{def:goodsequence} We say that a sequence of continuous functions $(g_i)_{i\in \N}$, $g_i:X\to [0,\infty)$, is a \emph{good sequence of functions}, if it satisfies the following properties:
\begin{enumerate}
\item Increasing: $g_i(x) \leq g_j(x)$ for each $i\leq j$ and all $x\in X$.
\item Positivity: For any bounded set $V\subset X$, there exists an $\eta_V>0$ so that $g_i(x) \geq \eta_V$ for every $i\in \N$ and every $x\in V$.
\item ``Goodness'': For any bounded set $V\subset X$, any $\delta,L>0$, and any sequence $(P_i)_{i\in \N}$ of discrete paths $P_i\subset V$ with $\lim_{i\to\infty}\mesh(P_i) =0$, and such that
\begin{enumerate}
\item $\int_{P_i} g_i \leq L$,
\item $\diam(P_i) \geq \delta$,
\end{enumerate}
there is a subsequence converging to a curve $\gamma$ in the sense of (\ref{eq:conv-discrete-cont}).
\end{enumerate}
\end{definition}

\begin{proposition}\label{prop:goodsequences}
Let $(X,d,\mu)$ be a complete separable metric measure space, where $\mu$ is a Borel measure that is positive and finite on $r$-balls with $0<r<\infty$. Assume $p\in[1,\infty)$ and let $g$ be given a lower semicontinuous function. For every $\epsilon>0$, there exists 
a good sequence of bounded Lipschitz continuous functions $(\tilde{g}_i)_{i\in\N}$  converging pointwise to a function $\tilde{g}$ that is a lower semicontinuous good function, and  so that for every bounded set $A\subset X$, there exists some $\eta_A>0$ so that 
\begin{equation}\label{eq:etaa-lb}
\tilde{g}(x)\ge g(x)+\eta_A,\qquad\text{for all $x\in A$,}
\end{equation}
and
\begin{equation}\label{eq:gtildeineq}
\int_X \tilde{g}^p d\mu \leq \int_X g^p d\mu + \epsilon.
\end{equation}
Moreover, if $K\subset X$ is a compact set on which $g|_K$ is bounded, then we can choose  $\tilde{g}$ so that $\tilde{g}|_K$ is bounded.
\end{proposition}

\begin{proof}
Let $x_0 \in X$ be arbitrary, and let $\epsilon \in (0,1)$. For simplicity, if $K$ is provided, scale the metric so that $K\subset B(x_0,1)$. Let $\psi_i(x)=\max(0,\min(i+1-d(x_0,x), 1))$ so that $\psi_i|_{B(x_0,i)}=1$ and $\psi_i|_{X\setminus B(x_0,i+1)}=0$. One directly observes that $\psi_i$ is Lipschitz for every $i\in \N$.

Define $E_i$ to be an increasing sequence of compact sets so that $\mu(B(x_0,i+1)\setminus E_i) \leq \epsilon^p 2^{-4pi}$. If the set $K$ is provided as in the `Moreover part' of the statement, then we choose $E_i$ so that $K\subset E_i$ for each $i$.  These sets $E_i$ can be constructed since $\mu$ is Radon. By lower semicontinuity, we may choose an increasing sequence of Lipschitz continuous bounded functions $g_i$ converging to $g$. The standard construction is to let $g_i(x):=\inf\{g(y)+id(x,y): y\in X\}$, see \cite[Proposition 4.2.2]{shabook}.

We modify these functions as follow:
\begin{equation}\label{eq:defseqg}
\tilde{g}_i(x) := g_i(x) +  \sum_{n=1}^i \left(n \min(1,d(x,E_n) ) + \frac{\epsilon}{8^n (\mu(B(x_0,n+1))+1)} \right)\psi_n(x).
\end{equation}
Note that $\tilde{g}_i$ is Lipschitz continuous and bounded as well.

Also, define
\[
\tilde{g}(x) := g(x) + \sum_{n=1}^\infty \left(n \min(1,d(x,E_n)) + \frac{\epsilon}{8^n (\mu(B(x_0,n+1))+1)}\right)\psi_n(x).
\]
Then, it holds that $\lim_{i\to \infty} \tilde{g}_i(x)=\tilde{g}(x)$ and $g(x) \leq \tilde{g}(x)$ for every $x\in X$. If the set $K$ was provided in the `Moreover part' of the proposition, then for every $x\in K$, we have $d(x,E_n)=0$ and $\psi_n(x)=1$, so $\tilde{g}\leq g + \epsilon$ is bounded on $K$.

We begin by verifying Inequality \eqref{eq:gtildeineq}. By Minkowski's Inequality,
\begin{align*}
\|\tilde{g}\|_{L^p(X)} & \leq \|g\|_{L^p(X)} +  \sum_{n=1}^\infty  \|n \min(1,d(\cdot,E_n)) \psi_n\|_{L^p(X)} + \left\|\frac{\epsilon}{8^n (\mu(B(x_0,n+1))+1)}\psi_n\right\|_{L^p(X)} 
\end{align*}
Note that
\[
 \min(1,d(\cdot,E_n)) \psi_n\le \ones_{B(x_0,n+1)\setminus
  E_n}\qquad\text{and}\qquad  \psi_n\le \ones_{B(x_0,n+1)}
\]
Therefore,
\begin{align*}
\|\tilde{g}\|_{L^p(X)} & \le  \|g\|_{L^p(X)} +\epsilon\sum_{n=1}^\infty (n 2^{-4n} + 8^{-n}) \le  \|g\|_{L^p(X)} +\epsilon.
\end{align*}
Raising both sides to the power $p$, applying the mean value theorem to the function $x\mapsto x^p$, and using $0<\ep<1$, we get that
\[
\|\tilde{g}\|_{L^p(X)} \le \|g\|_{L^p(X)}^p +\epsilon p\left(\|g\|_{L^p(X)}+1\right)^{p-1}.
\]
Finally, replacing $\epsilon$ with $\epsilon p^{-1}(\|g\|_{L^p(X)} +1)^{-(p-1)}$, yields the desired estimate (\ref{eq:gtildeineq}).

Also, let $\eta_i \defeq \epsilon 8^{-i}(\mu(B(x_0,i+1))+1)^{-1}$, then $\tilde{g}_i|_{B(x_0,i)}\geq g_i|_{B(x_0,i)} + \eta_i$ for every $i\in\N$. Further, we get that for any bounded set $A\subset X$ there is some $i$ so that $A\subset B(x_0,i)$ and so that $\tilde{g}|_A\geq g|_A+\eta_i$.

To show goodness for the sequence $\tilde{g}_i$. Let $L,\delta>0$ and let $A$ be a bounded set and consider any sequence $(P_i)_{i\in \N}\subset A$ of discrete paths $P_i$ with  $\lim_{i\to\infty}\mesh(P_i) =0$, such that
\begin{enumerate}
\item $\int_{P_i} \tilde{g}_i \leq L$,
\item $\diam(P_i) \geq \delta$.
\end{enumerate}
By passing to a subsequence, we can assume $\mesh(P_i) \leq \frac{1}{i}$. Since $P_i \subset A$, by (\ref{eq:etaa-lb}), we have $\tilde{g_i}|_{P_i} \geq \eta_A$ for some $\eta_A>0$. Let $L'=\frac{L}{\eta_A}$. Then 

\[
\len(P_i) \frac{1}{L'} = \int_{P_i} \frac{\eta_A}{L}  \leq \int_{P_i} \frac{1}{L}\tilde{g_i}|_{A} \leq 1.
\]

Thus, $\len(P_i) \leq L'$. By Lemma \ref{lem:compactness} it suffices to prove that for every $\tau\in (0,1)$ there is a compact set $K_\tau$ for which $d(P_i, K_\tau) \leq \tau$ for all $i\in\N$.  Without loss of generality, assume $\tau \in (0,\delta/2)$. 

Since $P_i \subset A$ for all $i\in \N$ and since $A$ is bounded, there is some $T$ so that $P_i \subset B(x_0,T)$ for all $i\in \N$. Choose $N=\lfloor 2^4 \max(L,1) / \tau^2 \rfloor+T+1$. Let $K_\tau = E_N \cup \bigcup_{i=1}^N P_i$. Then $K_\tau$ is compact, and it suffices to show that $\sup_{p\in P_i}d(p,K_\tau)\leq \tau$.
This is clear for $i=1,\dots, N$, thus consider $i>N$. 

Suppose for the sake of contradiction that there is a point $p_k^i\in P_i$ for some $k=0,\dots, n(i)$ with $d(p_k^i,K_\tau) > \tau$. Note that $\diam(P_i) \geq \delta>2\tau$ and $\mesh(P_i)\leq 1/i  \leq \tau/8$. Consider the maximal interval $[k_0,k_1]$ containing $k$ and so that the corresponding subpath $(p_l^i)_{l=k_0}^{k_1}$ stays in $B(p_k^i,\tau/2)$. In particular, for $0\leq k_0\leq l \leq k_1\leq n(i)$, we have $d(p_l^i,K_\tau) \geq \tau/2$. Furthermore, by maximality of the interval $[k_0,k_1]$, the path must exit the ball. Hence, $p_l^i \not \in B(p_k^i,\tau/2)$, for either $l=k_0-1$ or $l=k_1+1$, and
\begin{equation}\label{eq:largesubseq-1}
\sum_{l=k_0}^{k_1-1}d(p_l^i,p_{l+1}^i) \geq \tau/2-1/i \geq \tau/4.
\end{equation}

Now, take any index $l\in \{k_0,\dots, k_1\}$ and let $i\geq N$. Since $d(p_l^i,K_\tau) \geq \tau/2$ and $E_N \subset K_\tau$, we have 
\[
d(p_l^i,E_N)\geq d(p_l^i,K_\tau)\geq \tau/2.
\]
Further, since $p_l^i \subset B(x_0, T)\subset B(x_0,N)$ we have 
\[
N\min(1,d(p_l^i,E_N))\psi_N(p_l^i)\geq N\tau/2.
\]
Therefore, we get $\tilde{g}_i(p_l^i) \geq N \tau/2$ and thus by Inequality \eqref{eq:largesubseq-1} 
\[
L \geq \sum_{l=0}^{n(i)-1} \tilde{g}_i(p^i_l)d(p^i_l,p^i_{l+1}) \geq (\tau/4) (N \tau/2) > L.
\]
This is a contradiction, and thus, $d(p,K_\tau) \leq \tau$ for each $p\in P_i$ and all $i\in \N$.
\end{proof}

Finally, we formulate a result analogous to Lemma \ref{lem:lsccurve} and Lemma \ref{lem:lsc}, in the case of good sequences for functions.
\begin{lemma}\label{lem:lscgoodfunc} Suppose that $\{g_i\}_{i\in \N}$ is a good sequence of functions converging to $g$, and that $P_i$ is a sequence of discrete paths converging to a curve $\gamma$, then
\[
\int_\gamma g ds \leq \liminf_{i\to \infty} \int_{P_i} g_i.
\]
\end{lemma}

\begin{proof} By Lemma \ref{lem:lsccurve}, for any $l$ fixed, we have

$$\int_\gamma g_l ds \leq \liminf_{i\to \infty} \int_{P_i} g_l.$$

Since $\{g_i\}_{i\in \N}$ is an increasing sequence of functions, we get
$$\int_\gamma g_l ds \leq \liminf_{i\to \infty} \int_{P_i} g_i.$$
Sending $l\to\infty$ yields the claim.
\end{proof}

\subsection{Continuous ``almost'' upper gradients}

We will approximate an upper gradient by continuous functions. Recall, that a minimal $p$-weak upper gradient $g_f$ of a function $f\in N^{1,p}(X)$, is \emph{a priori}, only in $L^p(X)$. Lemma \ref{lem:lowersem} shows that, by introducing a small error $\epsilon>0$, we can find a lower semicontinuous function $g_\epsilon \in L^p(X)$ which is an actual upper gradient. We would like to replace $g_\epsilon$ with a continuous function. However, the upper gradient inequality \eqref{eq:ug} is only preserved if we approximate $g_\epsilon$ from \emph{above} by a function $h$. Further, it is impossible to approximate every $L^p(X)$-function from above by a continuous, let alone bounded, function. Fortunately, lower semicontinuous functions can be approximated \emph{from below} by a sequence of continuous bounded functions. This does not preserve \eqref{eq:ug}. However, it will preserve being an ``almost upper gradient'' in the following sense.

\begin{definition}\label{def:discreteupper}
Let $V$ be a closed set with $\mu(V)<\infty$ and let $C\subset V$ be a closed subset of $X$. A function $h$ is a {\it $(\delta,\Delta)$-discrete upper gradient} for $f$ on $(C,V)$ if for every discrete path $P=(p_0,\dots,p_n)$ with $\mesh(P)\leq \delta$, $P \subset V$, $p_0,p_n \in C$ and $\diam(\{p_0,\dots, p_n\})>\Delta$ we have
\[
|f(p_n)-f(p_0)|\leq \int_P h \,.
\]
\end{definition} 

Here, it is necessary to localize the condition to apply only to curves with large enough diameter,  which lie within a bounded set $V$, and which connect points in a closet set $C$. The first two of these are used to ensure compactness of the relevant families of curves. The final one is a bit more subtle, and is related to the fact that a Sobolev function may not be continuous, and $C$ should be thought of as a closed set such that $f|_C$ is continuous. In fact, the following lemma illustrates well the role of each of these assumptions.

\begin{lemma}\label{lem:farenough} Assume that $C,V\subset X$ are closed bounded sets with $C\subset V$. Let $M>0$ and $f:X\to [0,M]$ be a measurable function which is continuous on $C$. Let $g:X\to [0,\infty]$ be a lower semicontinuous upper gradient for $f$. Suppose that $\eta>0$ and $(g_i)_{i\in \N}$ is a good sequence of functions, which converges pointwise to a lower semicontinuous function $\tilde{g}$ with $\tilde{g}|_V > g|_V+\eta$, as constructed in Proposition \ref{prop:goodsequences}.
Then, for every $\Delta>0$ there exists an $N\in \N$ so that $g_i$ is a $(1/i,\Delta)$-discrete upper gradient for $f$ on $(C,V)$ for every $i\geq N$.
\end{lemma}
\begin{proof}
Arguing by contradiction, there exists $\Delta>0$ and an infinite subset $\mathbb{I}\subset \N$ so that for every $i\in \mathbb{I}$ there exists a path $P_{i}=(p_0^i,\dots, p_{n(i)}^{i})$ with $\mesh(P_i) \leq \frac{1}{i}$, $\diam(P_i) \geq \Delta$, $P_i \subset V$, $p_0^i, p_{n(i)}^i \in C$ and

\begin{equation}\label{eq:contradiction-discrete}
|f(p^i_{n(i)})-f(p^i_0)|> \int_{P_i} g_i \, .
\end{equation}

Since $|f|\le M$, we get  $\int_{P_i} g_i \leq 2M$ for each $i\in \mathbb{I}$. By Definition \ref{def:goodsequence} (3), there exists an infinite subset $\mathbb{J}\subset \mathbb{I}$, so that $(P_i)_{i\in \mathbb{J}}$ converges to a curve $\gamma$.

In particular, $\gamma(1)$ is a limit of the sequence $(p^i_{n(i)})_{i\in \mathbb{J}}$ , and thus $\gamma(1)\in C$. Similarly, $\gamma(0)$ is a limit of the sequence $(p^i_{0})_{i\in \mathbb{J}}$ , hence $\gamma(0)\in C$. Further, $\gamma \subset V$, since $V$ is closed, and $\diam(\gamma)\geq \Delta$ since $\diam(P_i) \geq \Delta$ for all $i\in \N$.

By sending $i\in \mathbb{J}$ to infinity in  Inequality \eqref{eq:contradiction-discrete}, using Lemma \ref{lem:lscgoodfunc} and the fact that $f|_C$ is continuous, we get 
\begin{equation}\label{eq:contradiction}
|f(\gamma(1))-f(\gamma(0))|\geq \int_{\ga} \tilde{g} ds. \,
\end{equation}
However, $\tilde{g}|_V> g|_V+\eta$, which contradicts the upper gradient inequality. Therefore the claim has been proved.
\end{proof}

\section{The case when $X$ is complete}
\label{sec:complete}

In this section we will prove versions of our main theorems when $X$ is a metric measure space that is complete and separable. We show that:
\begin{itemize}
\item That capacity is outer regular for sets $E$ with $\Cp_p(E)=0$;
\item different versions of the capacity are equal, namely $\Cp_p =\Cp_p^c=\Cp_p^{\lip}=\Cp_p^{(\lip,\lip)}$ under some weak hypothesis;
\item $C(X)\cap N^{1,p}(X)$ is dense in $N^{1,p}(X)$;
\item every function $f\in N^{1,p}(X)$ is quasicontinuous;
\item $\Cp_p$ is outer regular, and thus a Choquet capacity.
\end{itemize}
In the subsections that follow, we address each one of these claims in turn.

\subsection{Null capacity sets}

We will employ the following lemma for capacity. A set $E$ is said to be {\it $p$-exceptional}, if $\Mod_p(\Gamma_E)=0$, where $\Gamma_E$ is the collection of all rectifiable curves $\gamma$ for which $\gamma \cap E \neq \emptyset$.

\begin{lemma}\label{lem:capacity-exceptional}(\!\!\cite[Proposition 7.2.8.]{shabook}) Suppose that $(X,d,\mu)$ is a separable metric measure space, then a set $E\subset X$ satisfies $\Cp_p(E)=0$ if and only if $E$ is $p$-exceptional and $\mu(E)=0$.
\end{lemma}

Lemma \ref{lem:capacity-exceptional} is crucial when one wants to show that capacity is outer regular. The first step is to analyze sets with zero capacity. The proof of  Proposition \ref{prop:capacity-null-case} below follows closely that of \cite[Proposition  7.2.12]{shabook} -- except for the novel use of a good function.

\begin{proposition}\label{prop:capacity-null-case}  Suppose that $(X,d,\mu)$ is a complete separable metric measure space and let $E \subset X$ satisfy $\Cp_p(E)=0$. For any $\epsilon>0$ we have an open set $O$ s.t. $E \subset O$ and $\Cp_p(O)<\epsilon$.
\end{proposition}

\begin{proof}
Capacity is easily seen to be sub-additive and so it suffices to consider the case when $E$ is bounded. Thus, assume that $E \subset B(x_0,R)$ for some ball $B(x_0,R)$ with $x_0\in X, R>0$. Choose  an open set $V\subset B(x_0,R)$ so that $E\subset V$ and $\mu(V\setminus E)\leq \epsilon 2^{-p-1}$. 
By Lemma \ref{lem:capacity-exceptional}, the set $E$ is $p$-exceptional. Thus, $\Mod_p(\Gamma_E)=0$ and since $\Gamma(E,X\setminus V) \subset \Gamma_E$ we have $\Mod_p(\Gamma(E,X\setminus V))=0$. 
Let $g$ be an admissible function for $\Gamma(E,X\setminus V)$ with $\int_X g^p d\mu\le \epsilon 2^{-p-3}$. Lemma \ref{lem:goodfunc} provides a good function that is lower semicontinuous and admissible for $\Gamma(E,X\setminus V)$, with $g_\epsilon\ge g$  and $\int_X g_\epsilon^p d\mu \leq \epsilon 2^{-p-2}$.


Define $u(x) := \min(1,\inf_{\gamma:X\setminus V \mapsto x} \int_\gamma g_\epsilon ds )$. By Proposition \ref{prop:funconstr}, $u$ is lower semicontinuous, $u|_E= 1$, $u|_{X\setminus V}=0$, and $u$ has upper gradient $g_\epsilon$. Thus, $U=\{u>\frac{1}{2}\}$ will be an open set containing $E$ and $U\subset V$. Take $O=U$. Then, $\tilde{u}=2u \in N^{1,p}(X)$ and $\tilde{u}|_O \geq 1$. Therefore, from $\tilde{u}\leq 2\cdot \ones_{V\setminus E}$ we get
\[
\Cp_p(O) \leq \int_X |2u|^p d\mu + \int_X (2g_\epsilon)^p d\mu \leq \epsilon
\]
and the claim follows.
\end{proof}

\subsection{Different versions of capacity}
We now state and prove a version of Theorem \ref{thm:capacity}, when the space $X$ is assumed to be complete, rather than merely locally complete.
We will use Theorem \ref{thm:capacity-complete} later, in Section \ref{sec:choquet}, to prove the more general statement formulated  in Theorem \ref{thm:capacity}.
\begin{theorem}\label{thm:capacity-complete}
Let $(X,d,\mu)$ be a complete, bounded and separable metric measure space equipped with a Radon measure which is positive and finite on all balls. Let $E,F \subset X$ be two non-empty closed disjoint sets with $d(E,F)>0$, and let $p\in[1,\infty)$. Then
$$\Cp_p(E,F)=\Cp_p^c(E,F)=\Cp_p^{\lip}(E,F)=\Cp_p^{(\lip,\lip)}(E,F).$$
\end{theorem}   
\begin{corollary} \label{rmk:refinedversion} Let $E,F\subset X$ be two non-empty closed disjoint sets with $d(E,F)>0$, and let $p\in  [1,\infty)$. If $u\in N^{1,p}(X)$ is non-negative with $u|_E=0, u|_F=1$ and $g$ is an upper gradient for $u$ in $L^p(X)$, then there exists a sequence of functions $u_i \in N^{1,p}(X)$, which are locally Lipschitz, and which have locally Lipschitz upper gradients $h_{i}\in L^p(X)$, with $h_i \to g$ in $L^p(X)$. 
\end{corollary}
Note that $h_{i}$ need not be the minimal $p$-weak upper gradient of $u_i$.
\begin{proof}[Proof of Corollary \ref{rmk:refinedversion}]
The proof is the same as the one for Theorem \ref{thm:capacity-complete}, and is obtained by setting $h_{i}=(a_i)^{-1}g_i$ at the end of the proof. 
\end{proof}

\begin{proof}[Proof of Theorem \ref{thm:capacity-complete}] An infimum over a smaller set yields a larger value than an infimum over a larger set, and thus
$$\Cp_p(E,F)\leq \Cp_p^c(E,F)\leq \Cp_p^{\lip}(E,F) \leq \Cp_p^{(\lip,\lip)}(E,F).$$

Therefore, it suffices to prove $\Cp_p^{(\lip,\lip)}(E,F) \leq \Cp_p(E,F)$. If $\Cp_p(E,F)=\infty$, this is immediate. Thus, let us assume that $\Cp_p(E,F)<\infty$ and let $\epsilon>0$ be arbitrary. By definition of capacity, there  exists $u:X \to [-\infty,\infty]$  with $u|_E=0$ and $u|_F=1$ and an $L^p$-upper gradient $g:X\rightarrow[0,\infty]$ for $u$ with $\int g^p d\mu \leq \Cp_p(E,F)+\epsilon.$  By replacing $u$  with $\max(\min(u,1),0)$ we can assume that $u:X\to [0,1]$. Further, since $\mu(X)<\infty$, we get $u\in L^p(X)$ and, moreover, that $u\in N^{1,p}(X)$.

We have $\int_\gamma g \,ds \geq 1$ for each rectifiable $\gamma$ connecting $E$ to $F$. By Proposition  \ref{prop:goodsequences} there exists a $g_\epsilon \in L^p(X)$ which is lower semicontinuous with $g_\epsilon > g$ and $\int_X g^p d\mu \leq \int_X g^p d\mu + \epsilon$, and a good sequence of bounded and Lipschitz continuous non-negative functions $\{g_i\}_{i\in \N}$ that satisfy $g_i \nearrow g_\epsilon$ and $g_i\ge 0$.  Let
\begin{equation}\label{eq:uicapacity}
u_i(x) := \min\left(\inf\left\{\int_P g_i : P=(p_0, \dots, p_n), p_0 \in E, p_n=x, \mesh(P) \leq i^{-1}\right\},1\right).
\end{equation}
Note that $u_i:X\to [0,1]$, since we are taking a minimum with $1$. Further, $u|_E=0$ since for $x\in E$ we can use a constant path $P=(p_0)$. By Proposition \ref{prop:uppergradientpaths}, the function $g_i$ is an upper gradient for $u_i$.

We show first that the function $u_i$ is $M_i$-Lipschitz with $M_i=\max\{i, \sup_{x\in X} g_i(x)\}$. Note that $M_i<\infty$ since $g_i$ is bounded. To see the Lipschitz property, observe that  if $x,y\in X$ and $d(x,y)\geq \frac{1}{i}$, then since $0\le u_i\le 1$,
\[
|u_i(x)-u_i(y)|\leq 1 \leq M_id(x,y). 
\]
On the other hand, if $x,y\in X$ and $d(x,y) < \frac{1}{i}$, then any discrete path $P=(p_0,\dots, p_n)$ with $p_0 \in E, p_n = x, \mesh(P) \leq i^{-1}$, can be expanded to $P'=(p_0,\dots, p_n, y)$ with $\mesh(P')\leq i^{-1}$. It may be that $P'$ is not simple, as we require for paths. This occurs only if for some $i\in[0,n]$ we have $p_i=y$, and then we truncate $P'$ at such index.

This gives, $u_i(y) \leq \int_{P'} g_i \leq \int_P g_i + d(x,y) g_i(x)$. Infimizing over $P$ yields $u_i(y)\leq u_i(x) + M_i d(x,y)$. By symmetry, we get $|u(x)-u(y)|\leq M_id(x,y)$, which completes the proof of the Lipschitz bound.

Let $a_i = \inf_{x\in F} u_i(x)$. We show next that $\lim_{i\to \infty} a_i = 1$. Since $g_i \leq g_j$ is an increasing sequence of functions, the limit  $\lim_{i\to \infty} a_i$ exists. We obtain our claim via contradiction: Suppose that $\lim_{i\to \infty} a_i<1$. Then there would exist some $\delta>0$ so that $a_i < 1-\delta$ for every $i\in \N$. 

By definition, for every $i$, there exists a discrete path $P^i=(p_0^i, \dots, p_n^i)$ with $\int_{P_i} g_i < 1-\delta$ and with $p_0^i \in E, p_n^i \in F$ and $\mesh(P_i)<i^{-1}$. By the final condition, $\diam(P^i)\geq d(E,F)$ for each $i$. 

Since $g_i$ is a good sequence of functions, and since $X$ is bounded, there exists a subsequence $i_k$ so that $P_{i_k}\to \gamma$ for some curve $\gamma:[0,1]\to X$. Since $E$ and $F$ are closed, we conclude that $\gamma(0)\in E, \gamma(1)\in F$. By Lemma \ref{lem:lsccurve}, we have, for each $i\in \N$,
\[
\int_\gamma g_i \,ds \leq \liminf_{k\to\infty}\int_{P_{i_k}} g_i < 1-\delta.
\]
Sending $i\to \infty$, and with monotone convergence, we get $\int_\gamma g_\epsilon \,ds < 1-\delta$, which is a contradiction to the fact that $\int_\gamma g_\epsilon \,ds \geq \int_\gamma g \,ds \geq 1.$ Thus, our initial assumption was false, and $\lim_{i\to \infty} a_i = 1$.

Choose now $i$ so large that 
\[
\frac{\int_X g_\epsilon^p d\mu}{a_i^p} \leq \int_X g^p d\mu+2\epsilon.
\]

Then $\tilde{u}_i = \min(\frac{u_i}{a_i},1)$ is a Lipschitz function with the upper gradient $(a_i)^{-1}g_i$. Further, $\tilde{u}_i|_{E}=0$, $\tilde{u}_i|_{F}=1$, and 
\[
\int_X ((a_i)^{-1}g_i)^p d\mu \leq \int_X g^p d\mu+2\epsilon. 
\]
Thus, $\Cp_p^{\rm lip}(E,F) \leq \Cp_p(E,F)+2\epsilon$, and the claim follows since $\epsilon>0$ is arbitrary.
\end{proof}

The proof of the statement shows in fact slightly more. For future reference, we state this as a theorem.

 \begin{theorem}\label{thm:approximating-uppergrad}
Let $(X,d,\mu)$ be complete separable metric measure space with $\mu(X)<\infty$. Let $E,F \subset X$ be two non-empty closed disjoint sets with $d(E,F)>0$, and let $p\in[1,\infty)$. Then, for any $\epsilon>0$,  and for any ``admissible'' function $g\in L^p(X)$, i.e.,  so that $\int_\gamma g \, ds \geq 1$ for every $\gamma \in \Gamma(E,F)$,  there exists a locally Lipschitz function $g_\epsilon$ that is also admissible, meaning that $\int_\gamma g_\epsilon \, ds \geq 1$ for every $\gamma \in \Gamma(E,F)$, and such that $\|g-g_\epsilon\|_{L^p(X)}\leq \epsilon.$
\end{theorem}  

\subsection{Continuous functions are dense in Sobolev spaces}
Next, we prove the density of continuous functions in Newton-Sobolev spaces, in the case when the space $X$ is complete. Later, in Section \ref{sec:locallycomplete}, we will use Theorem \ref{thm:density-complete} below to extend the result to the case when $X$ is locally complete, which will thus give a proof for Theorem \ref{thm:capacity}.

\begin{theorem}\label{thm:density-complete} Let $(X,d,\mu)$ be  complete and separable metric measure space. Then $C(X) \cap N^{1,p}(X)$ is dense (in norm) in $N^{1,p}(X)$ for $p\in[1,\infty)$.
\end{theorem}

Given a function $f\in N^{1,p}(X)$ and $\ep>0$, we want to find a continuous Newton-Sobolev function $\tilde{f}$ on $X$ such that $\|f-\tilde{f}\|_{N^{1,p}(X)}\le \ep$. The idea of the proof is to consider an appropriately large compact set $K\subset X$ (see Equation (\ref{eq:large-K})), where $f|_K$ is continuous, and then to find an extension $\tilde{f}$ which is continuous everywhere, and which has controlled minimal $p$-weak upper gradient $g_{\tilde{f}}$. That is, our proof will be based on the following extension result of Whitney type. 
\begin{proposition}\label{thm:Extension}
Let $(X,d,\mu)$ be  complete and separable metric measure space. Let $f\in N^{1,p}(X)$ and let $g_*\in L^p(X)$ be an upper gradient. Suppose that $f|_{X\setminus B(x_0,R)}=0$ for some  $x_0\in X$, and $R>0$.  Suppose there is a compact set  $K \subset B(x_0,R)$ with $f|_K$ continuous and $g_*|_K$ bounded.

Then, for every $\epsilon>0$, there exists a function $\tilde{f}$ with:
\begin{enumerate}
\item $\sup_{x\in X}|\tilde{f}(x)|\leq \sup_{x\in K}|f(x)|$.
\item $\tilde{f}|_K = f|_K$ and  $\tilde{f}|_{X\setminus B(x_0,R)}=f|_{X\setminus B(x_0,R)}=0$.
\item $\tilde{f} \in N^{1,p}(X)\cap C(X)$.
\item $\int_{X\setminus K} g_{\tilde{f}}^p \, d\mu \leq \int_{X\setminus K} g_*^p\,d\mu+\epsilon$.
\end{enumerate}
\end{proposition}

We delay the proof of this extension result, briefly, in order to show how the density result follows from it.

\begin{proof}[Proof of Theorem \ref{thm:density-complete}.]
First, recall that by Lemma \ref{lem:densityN1pb} the space of bounded Newton-Sobolev functions with bounded support $N^{1,p}_b(X)$ is dense in $N^{1,p}(X)$. Next, we show that $C(X)\cap N^{1,p}_b(X)$ is dense in $N^{1,p}_b(X)$. If $f \in N^{1,p}_b(X)$, then there is a constant $M<\infty$ such that $|f|\leq M$ everywhere in $X$ and there is a ball $B(x_0,R)$ so that $f|_{X\setminus B(x_0,R)}=0$. Let $g\in L^p(X)$ be any upper gradient of $f$. Since $f=0$  in $X\setminus B(x_0,R)$, we can assume by modifying $g_*$ that $g_*|_{X\setminus B(x_0,R)}=0$. Indeed, this modification leaves \ref{eq:ug} invariant. 

Let $\epsilon>0$ be fixed, by using Lusin's theorem and the absolute continuity of integrals, choose a compact set $K \subset B(x_0,R)$ so that $f|_K$ is continuous, $g_*|_K$ is bounded and so that
\begin{equation}\label{eq:large-K}
\int_{X\setminus K} 2^{p+3}g^p \, d\mu + \mu(B(x_0,R)\setminus K)2^{p+1} M^p \leq \epsilon.
\end{equation}
This is possible, since $\mu$ is Radon, $g_*\in L^p(X)$, and $g_*= 0$ in $X\setminus B(x_0,R)$.

By Proposition \ref{thm:Extension} there exists a function $\tilde{f} \in N^{1,p}(X)\cap C(X)$ with  $\tilde{f}|_K = f|_K$  and 
\[
\int_{X\setminus K} g_{\tilde{f}}^p \, d\mu \leq \int_{X\setminus K} g_*^p \,d\mu + \epsilon 2^{-p-3}\leq \epsilon 2^{-p-2},
\]
where the last inequality follows by (\ref{eq:large-K}).
Furthermore, $\tilde{f}|_{X\setminus B(x_0,R)} =0$ and $|\tilde{f}|\leq M$ everywhere. 

Since $(f-\tilde{f})|_{K\cup (X \setminus B(x_0,R))}=0$, by Lemma \ref{lem:uppergradients} and subadditivity of minimal $p$-weak upper gradients, we have that $g_{f-\tilde{f}}\leq (g_{\tilde{f}}+g_*)\ones_{B(x_0,R)\setminus K}.$ Similarly, $|f-\tilde{f}|\leq 2M\ones_{B(x_0,R)}$. Thus, $\int_X g_{f-\tilde{f}}^p d\mu \leq \epsilon/2$. Also, we have $\int_X |f-\tilde{f}|^p\, d\mu\leq (2M)^p \mu(B(x_0,R)\setminus K).$ Therefore,
\[
\|f-\tilde{f}\|_{N^{1,p}(X)}^p \leq \int_X |f-\tilde{f}|^p \, d\mu  + \int_X g_{f-\tilde{f}}^p \, d\mu \leq \epsilon.
\]

\end{proof}

\begin{proof}[Proof of Proposition \ref{thm:Extension}.] This proof will take some detours and require some auxiliary results. The function $\tilde{f}$ is defined in Formula \eqref{eq:approximation} below. However, this definition depends on some technical choices, which are explained first. The crucial properties ensured by these choices are codified as lemmas. Once the function $\tilde{f}$ has been properly defined, we verify, one-by-one, the properties of the theorem. The proof ends at the end of this subsection by verifying the fourth property in the statement.

Fix a function $f\in N^{1,p}(X)$ and an arbitrary $\epsilon>0$. By assumption, 
\begin{equation}\label{eq:x0R}
\text{there is $x_0 \in X$ and $1\le R<\infty$, so that $f|_{X\setminus B(x_0,R)}=0$.} 
\end{equation}
Since $K$ is compact and $f$ is uniformly continuous on $K$, 
\begin{equation}\label{eq:M}
M:=\sup_{x\in K}|f(x)|<\infty.
\end{equation}
If $X\setminus B(x_0,R)\neq\emptyset$, assume, by scaling the metric $d$, that 
\begin{equation}\label{eq:scaling-d}
d(K, X \setminus B(x_0,R))\geq 1.
\end{equation} 

Let $g_*\in L^p(X)$ be the given upper gradient for $f$, which is assumed to be bounded on $K$. Thus, there is a constant $S<\infty$ so that $|\sup_{x\in K} g_*(x)|=S$. 
The result will be proven by constructing a function $\tilde{f}\in N^{1,p}(X) \cap C(X)$  so that the properties in Proposition \ref{thm:Extension} hold.

We construct $\tilde{f}$ together with an upper gradient $g \in L^p(X)$ for it.
Our choice of upper gradient $g$ is given in Lemma \ref{lem:choice-ug} below.
Let $\ep>0$ and let $g_\epsilon$ be a lower semicontinuous upper gradient for $f$ with $\int_{X} g_\epsilon^p d\mu \leq \int_{X}g_*^p d\mu +  \epsilon 2^{-5}$, as guaranteed by Lemma \ref{lem:lowersem}. By replacing $g_\epsilon(x)$ with $\min(g_\epsilon,S\ones_K+\infty \ones_{X\setminus K})$, we can assume that $g_\epsilon|_K$ is bounded by $|\sup_{x\in K}g_\epsilon(x)|=S$. Then, Proposition \ref{prop:goodsequences} applied to the function $g_\ep$ gives the following lemma.

\begin{lemma}\label{lem:choice-ug}
With $g_*$ and $g_\epsilon$ as defined above, there is a lower semicontinuous good function $g:X\to [0,\infty]$ so that  $g|_K$ is bounded, that admits a good sequence $\{g_i\}_{i\in \N}$ of bounded Lipschitz continuous functions  so that $g_i\nearrow g$ pointwise on $X$, and such that
\[
\int_{X} g^p d\mu \leq \int_{X} g_\ep^p d\mu +\ep 2^{-5} \leq \int_{X}g_*^p d\mu +  \epsilon 2^{-4}.
\]
Moreover, for every bounded set $V\subset X$ there exists an $\eta>0$ such that $g|_{V}>g_\epsilon+\eta$.  
\end{lemma}

We now introduce several auxiliary functions that require some motivation. By passing from $g_*$  to the functions $g_i$, we have gained continuity but at the price of losing the property that $g_i$ is an upper gradient for $f$. This loss forces us to make further choices, whose role we now briefly describe.

The construction of a good sequence of functions guarantees that $g_i$, for large $i$, is a discrete upper gradient in the sense of Definition \ref{def:discreteupper}. The discrete upper gradient property only holds for the closed sets $C$ and $V$, which we choose as follows. First,
\begin{equation}\label{eq:V}
V:=\overline{B(x_0,2R)}.
\end{equation} Note, that $V$ is the closure of the ball, and not the closed ball - although this makes little difference for the proof. Second,
\begin{equation}\label{eq:C}
C:=K \cup (\overline{B(x_0,2R)}\setminus B(x_0,R)).
\end{equation} 
Note that $f|_C$ is continuous, and $V$ is bounded, so the set $V$ localizes the argument. In the definition of $C$ we adjoin the annulus $\overline{B(x_0,2R)}\setminus B(x_0,R)$ to ensure later that our approximation will vanish outside of $B(x_0,R)$, see Figure \ref{fig:paths}.

The discrete upper gradient property includes two parameters $(\delta, \Delta)$, where the first controls the mesh-size and the second the diameter of discrete paths. As we decrease $\Delta$, we need to pass further into the sequence $g_i$ and decrease the mesh-size $\delta$. Due to this, we cannot construct an approximation using a single function $g_i$ or a single mesh-size $\delta$. Consequently, as we approach the set $K$, we will force the mesh-size to decrease, and the index $i$ to increase. This leads to a definition of auxiliary functions $\gap(x)$ and $G(x)$, where $\gap$ stands for the size of gaps, and $G$ will be a candidate gradient for the constructed function $\tilde{f}$. The idea of using different functions and meshes, which are fixed at dyadic length scales, comes from a Whitney-type extension argument. 

Finally, to force the property that $f|_C=\tilde{f}|_C$, we need to  control the behaviour off of $K$, and this involves the modulus of continuity  $\omega$ of $f|_K$. The modulus of continuity is used to define a ``penalty'' term $\bP$, which eventually will depend on the distance to $K$. The use of the penalty term is a bit similar to how one extends a uniformly continuous function off a subset to a uniformly continuous function on the entire space; see for example \cite{milman97}. 

The value of the approximation at a point $x\in X$ will be ultimately obtained by infimizing over discrete paths connecting $x$ to the closed set $C$, which have mesh-size controlled by $\gap(x)$ - these will be called $(x,\gap)$-admissible paths. The minimized function sums $G$ over such a path together with a penalty term and a term from $f$. The reader may now wish to glance at Equation \eqref{eq:approximation} to see how the three functions, $\bP,\gap,G$ are used. It may also be helpful to compare this to \eqref{eq:uicapacity}, or to the approximation and discussion found in \cite{seb2020} - where also a more detailed historical comparison is contained.

Denote by $\omega:(0,\infty) \to\R$ the modulus of continuity for $f|_K$, that is 
\begin{equation}\label{eq:mod-cont}
\omega(\delta) = \sup\{ |f(x)-f(y)| : d(x,y)\leq \delta, x,y\in K \}.
\end{equation}

\begin{lemma}\label{lem:auxiliaryfuns}
With $C,$ $V,$ and $K,$ as defined above in (\ref{eq:C}), (\ref{eq:V}), and, (\ref{eq:large-K}),
if $g_i$ is a sequence of good functions converging to $g$, as constructed in Lemma \ref{lem:choice-ug}, and $\omega$ is the modulus of continuity of $f|_K$, as defined in (\ref{eq:mod-cont}), then there exists an increasing sequence $i_n\in \N$, a gap function $\gap:X\to [0,\infty)$, a candidate upper gradient $G:X\to [0,\infty]$, and a penalty function $\bP:[0,\infty)\to [0,\infty)$, with the following properties given $n\in \N$:
\begin{enumerate}
\item[(i)] $g_{i_n}$ is an $(i_n^{-1}, 2^{-n})$-discrete upper gradient for $f$ on $(C,V)$;
\item[(ii)] the gap function $D$ satisfies:
\bi
\item[(a)] $\gap(x)\leq d(x,K)/4$ for every $x\in X$;
\item[(b)] $0<\gap(x)< i_3^{-1}$ for each $x\in X\setminus K$;
\item[(c)] $D(x)=\min(1/i_{n+2}, 2^{-n-3})$, if $2^{-n-1}\leq d(x,K) < 2^{-n}$, for $n\ge 1$;
\ei
\item[(iii)] the candidate upper gradient satisfies:
\bi
\item[(a)] $G(x)\geq g_{i_3}(x)$ for all $x$;
\item[(b)] for $n>3$, $G(x)\geq g_{i_n}$ if $d(x,K)\leq 2^{-n}$;
\item[(c)] for $n\ge 3$, $G(x)\leq g_{i_n}$ if $d(x,K) \geq 2^{-n-1}$;
\item[(d)] $G|_K=g|_K$.
\ei
\item[(iv)] the penalty function satisfies:
\bi
\item[(a)] $\bP(r)\geq 2M$ if $r \geq i_3^{-1}$; 
\item[(b)] $\lim_{r\to 0} \bP(r)=0$, but $\bP(r)\geq \omega(r)+\omega(2^{1-n})$ for $i_{n+1}^{-1}\leq r< i_{n}^{-1}$ and $n\geq 3$;
\item[(c)] $\bP(r)\geq \omega(r)$, for all $r>0$.
\ei
\end{enumerate}
\end{lemma}
\begin{proof}
By Lemma \ref{lem:farenough}, there is an increasing sequence $i_n$, for $n\in \N$, so that $g_{i_n}$ is an $(i_n^{-1}, 2^{-n})$-discrete upper gradient for $f$ on $(C,V)$. Define $\gap$ as a step function depending on dyadic length scales determined by the distance to $K$:
\[\gap(x) := \begin{cases} \min(1/i_3,1/8) & d(x,K) \geq 2^{-1} \\ \min(1/i_{n+2}, 2^{-n-3}) & 2^{-n-1}\leq d(x,K) < 2^{-n}, n\geq 1 \\  0 & x\in K.  \end{cases}\]
Similarly, define the candidate upper gradient piecewise using the good functions $g_i$ and their limit $g$.
\[G(x) := \begin{cases}  g_{i_3}(x) & d(x,K) \geq 2^{-3} \\ g_{i_n}(x) & 2^{-n-1}\leq d(x,K) < 2^{-n}, n\geq 3 \\ g(x)  & x\in K. \end{cases}\]
Finally, we define a penalty function depending on the modulus of continuity of $f$ on $K$.
\[\bP(r) := \begin{cases} 2M & \frac{1}{i_3} \leq r \\ \omega(2^{1-n})+\omega(r)& \frac{1}{i_{1+n}}\leq r< \frac{1}{i_{n}}, n\geq 3 \\  0 & r=0 \end{cases}.\]
Notice that we have $\lim_{r\to 0} \bP(r) = 0$. The properties of $\bP$ are direct to verify, once one notices that $2M\geq \omega(r)$ for all $r>0$. 

We let the reader verify that these definitions imply the properties stated above.

\end{proof}

Next, we describe the admissible discrete paths used in the extension. 

\begin{definition}\label{def:admissible}
With $V$, and $C$ defined as in (\ref{eq:V}) and (\ref{eq:C}),
a discrete path $P=(p_0, \dots p_n)$ is called $(x,\gap)$-{\it admissible} if $p_0\in C$, $p_n=x$, and $d(p_k,p_{k+1}) \leq \gap(p_k)$ for each $k=1,\dots, n-1$ and $p_1,\dots, p_n \in V$. 
\end{definition}
In other words, the length of the first step $d(p_0,p_1)$ can be arbitrary, but after this ``first jump'', the following steps are controlled by the function $\gap$. This ensures that there is always at least one $(x,\gap)$-admissible path for every $x\in V$, namely $P=(p_0,x)$ for any $p_0\in C$ (or $P=(p_0)$ if $x=p_0\in C$. This fact will be used to guarantee an upper-bound for the extension.

Note that Lemma \ref{lem:auxiliaryfuns} guarantees $\gap(x)\leq d(x,K)/4$ for all $x\in X$. This implies the following useful lemma. 
\begin{lemma} \label{lem:admissiblepath} 
With $K$ as defined  in (\ref{eq:large-K}),
let $\gap(x)\ge 0$ be a function such that $\gap(x)\leq d(x,K)/4$ for all $x\in X$. If $P=(p_0, p_1,\dots,p_n)$ is $(x,\gap)$-admissible and $p_1\not\in K$, then $p_l\not\in K$ for any $l\geq 1$. Alternatively, if $n\geq 1$ and $p_1\in K$, then $P=(p_0,p_1)$.
\end{lemma}
\begin{proof}
If $p_1\in K$, since $D\equiv 0$ on $K$,  $d(p_{l+1},p_l)=0$ for $l\ge 1$, so the path $P$ will stay at $p_1$. Else, if $p_1\not\in K$, then $d(p_{l+1},p_l) \le D(p_l) < d(p_l,K)/2$ for $l\ge 1$, hence $p_l\not\in K$ for $l\ge 1$.
\end{proof}
\begin{definition}\label{def:approximation} With the functions $\bP,G,\gap$ from Lemma \ref{lem:auxiliaryfuns} and $K,C,V\subset X$, as defined in (\ref{eq:large-K}), (\ref{eq:C}), (\ref{eq:V}), and $M,R \in (0,\infty),x_0 \in X$ defined in (\ref{eq:M}) and (\ref{eq:x0R}), with $d(K,X\setminus B(x_0,R))\geq 1$, as in (\ref{eq:scaling-d}). We set the extension of $f$ as follow: for $x\in V$, 
\begin{equation}\label{eq:approximation}
\tilde{f}(x) := \min\left(M, \inf_{P=(p_0,\dots, p_n=x)} \Phi(P) \right),
\end{equation}
 where the infimum is taken over all $(x,\gap)$-admissible paths $P$ and $\Phi(P)$ is defined as follows. 
\begin{equation}\label{eq:def-phip}
\Phi(P):= \bP(d(p_0,p_1)) + f(p_0) + \sum_{k=0}^{n-1}G(p_
k)d(p_k,p_{k+1}),
\end{equation}
when $n\geq 0$. When $n=0$, we define $\Phi(P):=f(p_0)$. Finally, for $x\not\in V$, set $\tilde{f}(x)=0$. 
\end{definition}

\begin{figure}
\centering
\begin{tikzpicture}[square/.style={regular polygon,regular polygon sides=4}]
\draw[very thick, draw,fill=black!10](0,0) circle (5cm);    
\draw[very thick, draw,fill=white](0,0) circle (3.5cm);  

\draw[very thick,fill=black!45] plot[smooth cycle] coordinates {(-1.5,0) (-0.5,0.5) (1.5,1.5) (0,2) (-2,0.5)};
\draw[very thick,fill=black!45] plot[smooth cycle] coordinates {(0,-1) (1,-1.5) (1.5,-0.5) (0.5,0)};

\begin{scope}[every node/.style={circle,thick, draw=black,fill=white!100!}]
\node (p) [scale=0.4]at (0.8,1.4) {};
\node (q) [scale=0.4]at (0,-4) {};
\node (s)[scale=0.4] at (0.8,-1) {};
\node (u)[scale=0.4] at (-4,0) {};
\node (t) [scale=0.4]at (0.1,1.5) {};
\end{scope}    

\begin{scope}[every node/.style={circle,thick, draw=black,fill=black!100!}]
\node (P) [scale=0.4]at (2.5,1) {};
\node (Q) [scale=0.4]at (0.1,-1.5) {};
\node (S)[scale=0.4] at (-0.8,1) {};
\node (U)[scale=0.4] at (-1.4,0.5) {};
\node (T) [scale=0.4]at (-1,1.6) {};
\node (R) [scale=0.4]at (3.5,2) {};
\end{scope}  

\node[below right] at (2.5,1) {$P$};
\node[left] at (0.1,-1.5) {$Q$};
\node[above right] at (-0.8,1) {$S$};
\node[above] at (-1.4,0.5) {$U$};
\node[left] at (-1,1.6) {$T$};
\node[above] at (3.5,2) {$R$};

\node[below right] at (0.5,1) {{\huge $K$}};

\begin{scope}[every node/.style={square,thin, draw=black,fill=white!100!}]

\node (p1) [scale=0.4]at (2,2.4) {};
\node (p2) [scale=0.4]at (2.4,2) {};
\node (p3) [scale=0.4]at (2.1,1.6) {};
\node (p4) [scale=0.4]at (2.8,1.4) {};

\draw[dashed, very thick] (p) -- (p1);
\draw[dashed] (p1) -- (p2) -- (p3) -- (p4) -- (P);

\node (t1) [scale=0.4]at (-0.7,3.1) {};
\node (t2) [scale=0.4]at (-1.2,3) {};
\node (t3) [scale=0.4]at (-2,2.5) {};
\node (t4) [scale=0.4]at (-1.5,2.4) {};
\node (t5) [scale=0.4]at (-1.8,2.1) {};
\node (t6) [scale=0.4]at (-1,2.2) {};
\node (t7) [scale=0.4]at (-1.4,1.9) {};
\node (t8) [scale=0.4]at (-0.8,1.8) {};

\draw[dashed, very thick] (t) -- (t1);
\draw[dashed] (t1) -- (t2) -- (t3) -- (t4) -- (t5) -- (t6) -- (t7) -- (t8) -- (T);

\node (q1) [scale=0.4]at (0.7,-3.1) {};
\node (q2) [scale=0.4]at (-0.4,-2.7) {};
\node (q3) [scale=0.4]at (0.2,-2.5) {};
\node (q4) [scale=0.4]at (-0.2,-2.3) {};
\node (q5) [scale=0.4]at (0.1,-2.1) {};
\node (q6) [scale=0.4]at (0.3,-1.7) {};

\draw[dashed, very thick] (q) -- (q1);
\draw[dashed] (q1) -- (q2) -- (q3) -- (q4) -- (q5) -- (q6) -- (Q);

\draw[dashed, very thick] (s) -- (S);

\draw[dashed, very thick] (u) -- (U);
\end{scope}

\end{tikzpicture}
\caption{The figure shows six different admissible paths starting at various points in $C$ represented by  white circles and ending at the points $P,Q,R,S, T$ and $U$. The set $C$ is the lighter gray annulus together with the compact set $K$, represented by the two darker gray islands in the center. Squares indicate the points along the paths and dashed segments the jumps that we imagine occurring in between.\\
- The points $P$ and $Q$ show how a path can either start in $K$ or in the annular region. \\
- The point $R$ shows that a path contained in $C$ can have zero length. \\
- The path ending at $S$ shows that one can jump between points in $K$, but must stop there.\\ 
- Similarly, the path ending at $U$ makes one jump from $C\setminus K$ to $K$. \\
- Finally, the path ending at $T$ depicts how a path that at some point leaves $K$ can never return, but can get very close. \\
The first unrestricted jump of each path is bolded, and note that all the paths must be contained in the bounded set $V$ which is a ball containing the full figure. }\label{fig:paths}
\end{figure}
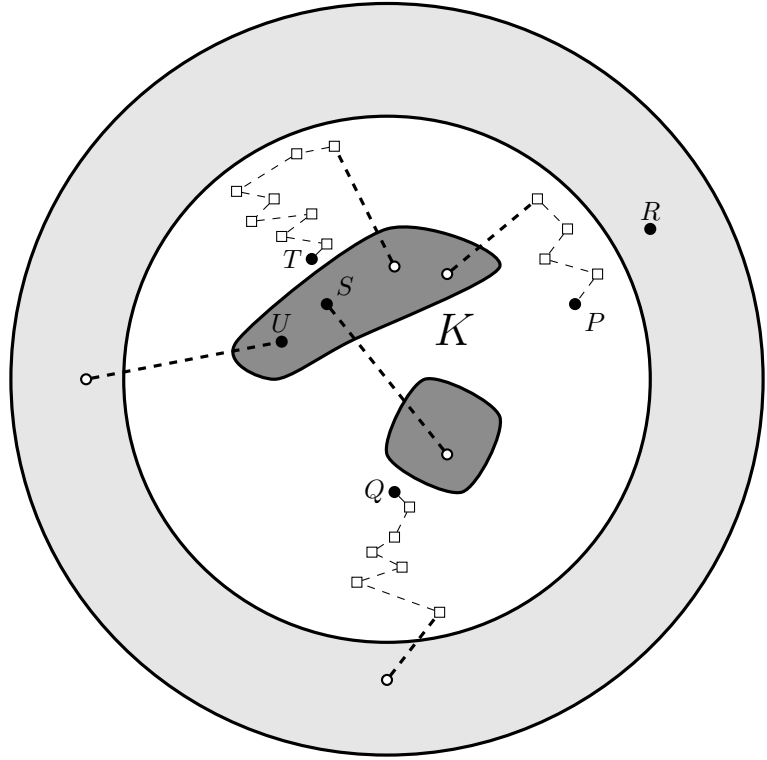

Examples of admissible paths are depicted in Figure \ref{fig:paths}. We now show that  $\tilde{f}$ satisfies properties 1--4 in the statement of Proposition \ref{thm:Extension}.
\begin{flalign}\label{eq:property-one}
\textbf{Property (1): }  \sup_{x\in X}|\tilde{f}(x)|\leq \sup_{x\in K}|f(x)|. && 
\end{flalign}
\noindent  Recall that $M=\sup_{x\in K} |f(x)|$. We want to show that $\tilde{f}:X\to[-M,M]$. First,  from Definition \ref{def:approximation}, it follows that $\tilde{f}(x)\leq M$ for each $x\in X$. On the other hand, since $G\geq 0$ and $\bP\geq 0$, $\Phi(P)\geq f(p_0)$ for every path $P=(p_0,\dots, p_n)$. Thus, $\inf_P \Phi(P)\geq \inf_{x\in C}f(x)=\inf_{x\in K}f(x)\geq -M$, and $\tilde{f}(x)\geq -M$ for $x\in V$.  For $x\not\in V$ we have $\tilde{f}(x)=0$, and the claim follows.

\begin{flalign}\label{eq:property-two}
 \textbf{Property (2): } \tilde{f}|_K = f|_K\quad\text{and}\quad \tilde{f}|_{X\setminus B(x_0,R)}=f|_{X\setminus B(x_0,R)}=0. &&
\end{flalign}
Take an arbitrary $x\in K \cup (X \setminus B(x_0,R))=C\cup(X\setminus V)$, where $C, V$, and $K$ are defined above in (\ref{eq:large-K}), (\ref{eq:V}), and (\ref{eq:C}).
We will show that $\tilde{f}(x)=f(x)$.  If $x\in X\setminus V$, then, by definition, $\tilde{f}(x)=0=f(x)$. Thus, we can assume that $x\in C$. The path $P=(x)$ is $(x,\gap)$-admissible and so we have $\tilde{f}(x) \leq \Phi(P)=f(x)$. We want to prove the opposite inequality, and for that we will separate the two cases when $x\in K$ and when $x\in C\setminus K$.

$\bullet$ First, consider the case $x\in K$. It suffices to prove $\Phi(P)\geq f(x)$ for every $(x,D)$-admissible path $P$. This is clear if $P=(x)$, and thus we can assume that $P=(p_0,\dots,p_n=x)$ with $n\geq 1$. If $p_1\not\in K$, then $p_1\ne x$ and Lemma \ref{lem:admissiblepath} shows that the path will never reach $x$. Therefore, $p_1$ must be in  $K$. Then, Lemma \ref{lem:admissiblepath} again shows that $p_1=x$ and $n=1$, so $P=(p_0,p_1=x)$. In other words, we are reduced to considering a path with only one jump.
There are two possibilities: 
\bi
\item[---] If $p_0$ is in the annulus $C \setminus K$, the path $P$ is illustrated by the path ending at the point $U$ in Figure \ref{fig:paths}. Then, since $p_0\not\in B(x_0,R)$, the normalization (\ref{eq:scaling-d}) implies that $d(p_0,x)\geq 1$. By Property (iv)(a) of Lemma \ref{lem:auxiliaryfuns}, the penalty on the first jump satisfies $\bP(d(p_0,p_1))\geq 2M$. Therefore, from the definition of $\Phi(P)$ in (\ref{eq:def-phip}), we get that $\Phi(P)\geq \bP(d(p_0,p_1)) \geq 2M \geq f(x)$. 
\item[---] If, on the other hand, $p_0 \in K$, then  
the path $P$ is illustrated by the path ending at the point $S$ in Figure \ref{fig:paths}. In this case, we estimate using the modulus of continuity:
\begin{align*}
\bP(d(p_0,p_1)) & \geq \omega(d(p_0,p_1)) & \text{(Property (iv)(c) of Lemma \ref{lem:auxiliaryfuns})}\\
& \geq |f(p_1)-f(p_0)| & \text{(by (\ref{eq:mod-cont}))}
\end{align*}
Thus, from (\ref{eq:def-phip}), we get $\Phi(P)\geq f(p_0) + \bP(d(p_0,p_1)) \geq f(x)$.
\ei

$\bullet$ Now, consider the case that $x\in C \setminus K$, and $P=(p_0,\dots, p_n)$ is any $(x,\gap)$-admissible path. We need to show that $\Phi(P)\geq f(x)=0$ for every such path. Recall that $(x,D)$-admissible paths start in $C$ and end at $x$. If $p_0\not\in K$, then $f(p_0)=0$, and, from (\ref{eq:def-phip}), we get $\Phi(P)\geq f(p_0)=0$. 
Thus, we are left to consider the case when $p_0\in K$ and $x\in C\setminus K$.
Let $i_n$ be the sequence constructed in Lemma \ref{lem:auxiliaryfuns}.
If the first jump satisfies $d(p_0,p_1)\geq i_1^{-1}>i_3^{-1}$, then by Lemma \ref{lem:auxiliaryfuns} (iv)(a), $\bP(d(p_0,p_1))\geq 2M$ and $f(p_0)+\bP(d(p_0,p_1))\geq 0=f(x)$. Therefore, we can assume that $d(p_0,p_1)\leq i_1^{-1}$. By Property (ii)(b) of Lemma \ref{lem:auxiliaryfuns}, we have $\gap(p_k)\leq i_1^{-1}$, for $k\ge 1$. Thus, $d(p_k,p_{k+1})\leq \gap(p_k)\leq i_1^{-1}$ for all $k\geq 1$. In particular, $\mesh(P)\leq i_1^{-1}$. Since $x\in C \setminus K$, the normalization (\ref{eq:scaling-d}) implies that $d(x,K)\geq 1$ and therefore $\diam(P)\geq 1$. Finally, by Property (iii)(a)  of Lemma \ref{lem:auxiliaryfuns}, we have that $G\geq g_{i_1}$. Recall that $g_{i_1}$ is a discrete $(i_1^{-1}, 2^{-1})$-upper gradient for $f$ with respect to $(C,V)$. By Definition \ref{def:discreteupper} and Definition \ref{def:admissible}, since $\mesh(P)\leq i_1^{-1}$, $p_0, p_n\in C$, $P \subset V$, and $\diam(P)\geq 1$,  we get
\[
\sum_{k=0}^{n-1} g_{i_1}(p_k)d(p_k,p_{k+1}) = \int_P g_{i_1} \geq |f(p_n)-f(p_0)|.
\]
In particular, $\Phi(P)\ge f(p_0) + \sum_{k=1}^{n-1}G(p_
k)d(p_k,p_{k+1}) \geq f(p_n)=0$. 

Finally, the inequality $\tilde{f}(x)\geq f(x)$ follows by infimizing over discrete paths $P$.
\begin{flalign}\label{eq:property-three}
\textbf{Property (3): } \tilde{f} \in N^{1,p}(X)\cap C(X).&&
\end{flalign}
To prove $\tilde{f} \in  N^{1,p}(X)\cap C(X)$, we proceed in a few stages. First, we show local Lipschitz continuity and an upper gradient property in the complement of $K$.

\begin{lemma}\label{lem:loclip}
The function $\tilde{f}$ is locally Lipschitz in $X \setminus K$ with upper gradient the function $g$ restricted to $X\setminus K$, where $g$ is defined as in Lemma \ref{lem:choice-ug}.
\end{lemma} 

\begin{proof} Recall that $D(x)>0$ for $x\in X\setminus K$, by Lemma \ref{lem:auxiliaryfuns} property (ii)(b). If $x,y \in X \setminus K$ and if $d(x,y) \leq \min\{D(x),D(y)\}$, then we claim that 
\begin{equation}\label{eq:lipest}
|\tilde{f}(x)-\tilde{f}(y)|\leq \max\{G(x),G(y)\}d(x,y). 
\end{equation}
Suppose for the moment, that this is true. Then, we can show that $g$ is an upper gradient for $\tilde{f}$ in $X\setminus K$. Indeed,
let $x\in X\setminus K$ and fix $n\ge 1$ so that $2^{-n}\leq d(x,K)$. 
Again, let $i_n$ be the sequence constructed in Lemma \ref{lem:auxiliaryfuns}.
Take any $r< \min(2^{-n-3}, 1/i_{n+2})/2$. Then, for every $y \in B(x,r)$ we have $d(y,K) > 2^{-n-1}$ and, by Lemma \ref{lem:auxiliaryfuns} properties (ii)(c) and (iii)(c), we have and $D(y) \geq 2r$   and $G(y) \leq g_{i_{n}}(y) \leq C(i_{n})$, where $C(i_{l})$ is the supremum of the bounded function $g_{i_{l}}$ for $l\in \N$. Thus Inequality \eqref{eq:lipest} implies that $\tilde{f}|_{B(x,r)}$ is $C(i_{n})$-Lipschitz. Using this local Lipschitz property and compactness, if $\gamma:[0,1] \to X\setminus K$ is any rectifiable curve then $\tilde{f}\circ\gamma$ is Lipschitz, and $d(\gamma,K)>2^{-n-1}$ for some $n$. Therefore, by Lemma \ref{lem:auxiliaryfuns} property (iii)(c) again, we have $G\circ \gamma \leq g_{i_n}\circ\gamma$ and that $|\tilde{f}(\gamma(t))-\tilde{f}(\gamma(s))|\leq \max\{g_{i_n}(\gamma(t)),g_{i_n}(\gamma(t))\}d(\gamma(s),\gamma(t))$ for any $s,t\in [0,1]$ with $d(\gamma(s),\gamma(t)) \leq \min\{D(\gamma(s)),D(\gamma(t))\}$. Following similar arguments as in \cite[Lemma 4.7]{sha00}, this together with the continuity of $g_{i_n}$ and Lemma \ref{lem:auxiliaryfuns} property (ii)(c) yields
\[
|\tilde{f}(\gamma(0))-\tilde{f}(\gamma(1))|\leq \int_0^1 |(\tilde{f}\circ \gamma)'|dt \leq \int_\gamma g_{i_n} \, ds.
\]
Since $g_{i_n}\leq g$, then $g$ is an upper gradient for $\tilde{f}$ in $X\setminus K$.

Next, we prove Inequality \eqref{eq:lipest}. Fix $x,y \in X \setminus K$ with $d(x,y) \leq \min\{D(x),D(y)\}$. By Property (ii)(b) of Lemma \ref{lem:auxiliaryfuns} we have $\min\{D(x),D(y)\}\leq 1$. Recall that,  by (\ref{eq:scaling-d}), $R>1$.  If $x$ or $y$ is in $X\setminus B(x_0,2R)$, then $d(x,y) \leq 1$ and both $x,y\not\in B(x_0,R)$. Thus, by Property 2 in (\ref{eq:property-two}), $\tilde{f}(x)=\tilde{f}(y)=0$ and inequality \eqref{eq:lipest} is immediate. 

We are left to consider the case when $x,y\in B(x_0,2R)$ and Formula \eqref{eq:approximation} gives the values of the function $\tilde{f}$ at $x$ and $y$. By symmetry, it suffices to show that $\tilde{f}(x) \leq \tilde{f}(y)+G(y)d(x,y)$. If $\tilde{f}(y)=M$, the claim follows by the definition of $\tilde f$. Otherwise, we have $\tilde{f}(y)=\inf_{P} \Phi(P)$, where $P=(p_0,\dots,p_n)$ runs through all $(y,\gap)$-admissible paths. Let $P$ be any $(y,\gap)$ admissible path. If $x\in P$, then by truncating $P$, we obtain an $(x,\gap)$ admissible sub-path, and $\tilde{f}(x)\leq \Phi(P)$ by definition. If $x\not \in P$, the augmented path $P'=(p_0,\dots, p_n,x)$ is $(x,\gap)$ admissible, since $d(p_n,x)=d(y,x)\leq \min\{D(x),D(y)\}\leq \gap(p_n)$. Thus,
\[
\tilde{f}(x) \leq \Phi(P')=\bP(d(p_0,p_1)) + f(p_0) + \sum_{k=0}^{n-1} G(p_k)d(p_k,p_{k+1}) + G(y)d(x,y)=\Phi(P)+ G(y)d(x,y).
\]
Infimizing over all discrete $(y,\gap)$-admissible paths $P$ now yields $\tilde{f}(x)\leq \tilde{f(y)}+G(y)d(x,y)$, and thus the claim.
\end{proof}

Next, we prove continuity of $\tilde f$ on all of $X$.

\begin{lemma}\label{lem:fcont}
The function $\tilde{f}: X\longrightarrow \R$ is continuous.
\end{lemma}
\begin{proof}
By Lemma \ref{lem:loclip}, the function $\tilde{f}|_{X\setminus K}$ is continuous. Also, $\tilde{f}|_K=f|_K$ is continuous by assumption. Recall that $K$ is closed. Thus to prove the lemma, it suffices to prove sequential continuity at points of $\partial K$, with the sequence approaching from $X\setminus K$. Fix $x\in \partial K$ and a sequence $x_i\to x$ with $x_i\not\in K$ for each $i\in \N$. Since $\partial K\subset B(x_0,R)$, we may assume that $x_i\in B(x_0,R)$ for all $i$ by passing to the tail of the sequence. Since the first jump is free, for every $i\in \N$, the path $P=(x,x_i)$ is an $(x_i,\gap)$-admissible path. Thus, by Property (iv)(d) of Lemma \ref{lem:auxiliaryfuns}, we have
\[
\tilde{f}(x_i) \leq f(x)+\bP(d(x,x_i))+g(x)d(x,x_i).
\]
By Property (iv)(b) of Lemma \ref{lem:auxiliaryfuns} and the fact that $g|_K$ is bounded by Lemma \ref{lem:choice-ug}, we have $\limsup_{i\to\infty}\tilde{f}(x_i) \leq f(x)$. So, it suffices to show that $\liminf_{i\to\infty} \tilde{f}(x_i)\geq f(x)$. Indeed, by passing to a subsequence, it suffices to assume that the limit $\lim_{i\to \infty} \tilde{f}(x_i)$ exists and then to show that 
\begin{equation}\label{eq:mainclaim}\lim_{i\to \infty} \tilde{f}(x_i) \geq f(x).
\end{equation} In the following, we will analyze several sub-cases depending on the values of $\tilde{f}$ and the constructed paths. Eliminating each sub-case will reduce the problem to a simpler situation. In the following, WLOG is short for ``Without loss of generality''.

\vskip.3cm

\noindent \textbf{Reduction 1: WLOG $\tilde{f}(x_i)<M$ for infinitely many $i$.} If for all but finitely many $i$ we have $\tilde{f}(x_i)=M$, the claim \eqref{eq:mainclaim} follows from the definition of $M$. 

Thus, we may pass to a subsequence, where $\tilde{f}(x_i)<M$ for every $i\in\N$. By definition of $\tilde{f}$ in \eqref{eq:approximation}, we may find discrete paths $P_i=(p^i_0,\dots, p^i_{n(i)})$ which are $(x_i,\gap)$-admissible and for which
\begin{equation}\label{eq:inthelimit}
\lim_{i\to\infty} \tilde{f}(x_i) = \lim_{i\to\infty} \Phi(P_i) = \lim_{i\to\infty}\left\{ \bP(d(p^i_0,p^i_1)) + f(p^i_0) + \sum_{k=0}^{n(i)-1} G(p^i_k)d(p^i_k,p^i_{k+1})\right\},
\end{equation}
and $\Phi(P_i)<M$. Note that $n(i)>0$, because $x_i\in B(x_0,R)\setminus K$, and hence is not in $C$.

\vskip.3cm

\noindent \textbf{Reduction 2:  WLOG the points $p^i_0$ do not converge to $x$.} If $\lim_{i\to\infty}p^i_0 = x$, then  we get that $\lim_{i\to\infty} \Phi(P_i)\geq \lim_{i\to\infty} f(p^i_0) = f(x)$, as desired, because $p^i_0$ and $x$ are in $C$ and $f|_C$ is continuous.  

Thus, by passing to some subsequence we are left to consider the case that $\lim_{i \to \infty} d(p^i_0,x)=\Delta$ for some $\Delta>0$. By further passing to a subsequence we can ensure $\Delta/2 \leq d(p^i_0,x) \leq 2\Delta$, so that $\diam(P_i) \geq \Delta/2$ for each $i\in \N$.
By passing to another subsequence,  since $\lim_{i \to \infty} p_{n(i)}^i = \lim_{i \to \infty} x_i = x$, we can assume that  that 
\begin{equation}\label{eq:large-diam-path}
d(p_{n(i)}^i,x) \leq \min\{\Delta/2, i_3^{-1}\} \leq \diam(P_i)\qquad\text{for all $i\in\N$}.
\end{equation}
This allows us to compare the values of $f(x)$ and $f(p_0^i)$  by considering the augmented discrete path $P'_i=(p_0^i,\dots, p_{n(i)}^i, x)$. 

At this point we may picture the path $P_i$ as the path corresponding to the point $T$ in Figure \ref{fig:paths}. The path $P'_i$ is obtained by augmenting $P_i$ with a jump to $x$. Hence,  $P_i'$ is no longer admissible.

\vskip.3cm

\noindent \textbf{Reduction 3: WLOG the first jump  $d(p^i_0,p^i_1)$ is less than $i_3^{-1}$.} If not, by Lemma \ref{lem:auxiliaryfuns} (iv)(a), $\bP(d(p_0^i,p^i_1))\geq 2M$. However, this contradicts the fact that $\Phi(P_i)<M$. Therefore, we must thus have $d(p^i_0,p^i_1) \leq \frac{1}{i_3}$ for all $i\in \N$. 

Recall that Lemma \ref{lem:auxiliaryfuns} (ii)(a-b) gives $\gap(x)\leq i_3^{-1},$ for all $x \in X$. Thus, since $P_i$ is $(x_i,\gap)$-admissible, and $d(p_{n(i)}^i,x)\leq i_3^{-1}$, we get $\mesh(P_i')\leq i_3^{-1}$. 

\vskip.3cm

\noindent \textbf{Reduction 4: WLOG eventually the diameter $\diam(P_i')$ is less than $2^{-3}$.}  If $\diam(P_i')\geq 2^{-3}$, for infinitely many $i\in\N$, we may pass to a subsequence where this property holds. Then, since $g_{i_3}$ is $(1/i_3,2^{-3})$-discretely admissible for $f$ with respect to $(C,V)$, we can apply the admissibility condition to the path $P'$ and we get
\begin{align}\label{eq:Pprimesum}
|f(x)-f(p_0^i)| &\leq  g_{i_3}(p_{n(i)}^i)d(p_{n(i)}^i,x)+\sum_{k=0}^{n(i)-1} g_{i_3}(p_k^i)d(p_k^i,p_{k+1}^i)  \nonumber \\
&\leq g_{i_3}(p_{n(i)}^i) d(p_{n(i)}^i,x)+\sum_{k=0}^{n(i)-1} G(p_k^i)d(p_k^i,p_{k+1}^i).& \text{(by Lemma \ref{lem:auxiliaryfuns} (iii)(a))}
\end{align}
Thus, 
\[
\Phi(P_i)=\bP(d(p^i_0,p^i_1)) + f(p^i_0) + \sum_{k=0}^{n(i)-1} G(p^i_k)d(p^i_k,p^i_{k+1})\geq f(x) - g_{i_3}(p_{n(i)}^i) d(p_{n(i)}^i,x),
\] 
and by sending $i\to\infty$ along such a subsequence and noting that $g_{i_3}$ is continuous and bounded, we get that $\lim_{i\to\infty} \Phi(P_i)\geq f(x)$. Therefore, in this case, $\lim_{i\to\infty} \tilde{f}(x_i)\geq f(x)$ using Equation \eqref{eq:inthelimit}.

Thus, by the last reduction and by (\ref{eq:large-diam-path}), we can assume that 
\[
\frac{\Delta}{2} \leq \diam(P'_i)\leq 2^{-3}.
\]
In particular,
if we let $L\in \Z$ be such that $2^{-L}\leq \Delta < 2^{1-L}$, then $L\geq 3$. Since $p_{n(i)}^i$ is converging to $x$, by passing to the tail, we can assume 
$d(p_{n(i)}^i,x) \leq \min\{\Delta/2, i_{L+1}^{-1}\}$.

Let $l_i\in \N$  be such that  $3\le l_i\leq L$ and $2^{-l_i-1}\leq \diam(P'_i) \leq 2^{-l_i}$. By the pigeonhole principle, we can pass to a sub-sequence with $l_i=l$ for all $i\in \N$.
Given $l\in \N$, which controls the size of the diameter, we now want to control the mesh-size of the path $P_i'$.

\vskip.3cm

\noindent \textbf{Reduction 5: WLOG for all but finitely many indices $i$, we have $d(p_0^i,p^i_1)\leq i_{l+1}^{-1}$,} where $i_n$ is the bound for the mesh-size defined in Lemma \ref{lem:auxiliaryfuns}. If not, then $d(p_0^i,p^i_1)\geq i_{l+1}^{-1}$ for infinitely many $i\in \N$. For such indices $i$, we have $\bP(d(p_0^i,p^i_1))\geq \omega(2^{1-l})$ by Property (iv)(b) of Lemma \ref{lem:auxiliaryfuns}. Thus, 
\begin{align*}
\Phi(P_i) & \geq \bP(d(p_0^i,p^i_1))+f(p_0^i) \\
& \geq \omega(2^{1-l})+f(p_0^i) \\
& \geq \omega(\diam(P'_i))+f(p_0^i)  & \text{(since $2^{1-l}\ge \diam(P'_i)$)}\\
& \geq \omega(d(x,p_0^i))+f(p_0^i)   & \text{(since $x,p_0^i\in P'_i$)}\\
&\geq f(x)
\end{align*}
Letting $i\to \infty$ along the given subsequence gives the claim. Therefore, we can assume by passing to the tail that $d(p_0^i,p^i_1)\leq i_{l+1}^{-1}$ for all $i\in \N$.

\vskip.3cm

\noindent \textbf{End of proof of Lemma \ref{lem:fcont}:} By the reductions described above, after passing to a subsequence, 
\begin{equation}\label{eq:first-and-last-step}
d(p_0^i,p^i_1)\leq i_{l+1}^{-1}\qquad\text{and}\qquad d(p_{n(i)}^i,x)\leq i_{L+1}^{-1}\leq i_{l+1}^{-1}
\end{equation}
For $k=1, \dots, n(i)-1$, we have $d(p_k^i,K)\leq \diam(P'_i)\leq 2^{-l}$. Since $P_i$ is $(x,D)$-admissible, by Property (ii)(c) of  Lemma  \ref{lem:auxiliaryfuns},  
\begin{equation}\label{eq:middle-steps}
d(p_k^i,p_{k+1}^{i}) \leq \gap(p_k^i) \leq \frac{1}{i_{l+1}}.
\end{equation} 
Combining (\ref{eq:first-and-last-step}) and (\ref{eq:middle-steps}), we get $\mesh(P_i') \leq 1/i_{l+1}$ and $\diam(P_i') \geq 2^{-l-1}$. Finally, since $g_{i_{l+1}}$ is $(i_{l+1}^{-1},2^{-l-1})$-discretely admissible for $f$ for $(C,V)$ we get Inequality \eqref{eq:Pprimesum} with $i_{l+1}$ replacing $i_3$.

Therefore, 
\[
\Phi(P_i)=\bP(d(p^i_0,p^i_1)) + f(p^i_0) + \sum_{k=0}^{n(i)-1} G(p^i_k)d(p^i_k,p^i_{k+1})\geq f(x) - g_{i_{l+1}}(p_{n(i)}^i) d(p_{n(i)}^i,x),
\]
and the claim follows from Equation \eqref{eq:inthelimit} by sending $i\to\infty$ and noting that $g_{i_{l+1}}$ is continuous.




\end{proof}

Next, we quickly get the Sobolev property. 

\begin{lemma}\label{lem:Sob} The function $\tilde{f}\in N^{1,p}(X)$ and $g$ is its upper gradient.
\end{lemma}
\begin{proof}
Recall that $f\in N^{1,p}(X)$, $g$ is an upper gradient for $f$, $f|_K=\tilde{f}|_K$ by Property (\ref{eq:property-two}), $\tilde{f}\in C(X)$ by Lemma \ref{lem:fcont}, and $g$ is an upper gradient for $\tilde{f}$ in $X\setminus K$ by Lemma \ref{lem:loclip}. Thus, Proposition \ref{prop:extension} applied to $f,g,\tilde{f}$ and $K$ shows that $g$ is an upper gradient for $\tilde{f}$ and that consequently $\tilde{f}\in N^{1,p}(X)$. 
\end{proof}
This establishes Property (\ref{eq:property-three}).
\begin{flalign}\label{eq:property-four}
\textbf{Property (4): } \int_{X\setminus K} g_{\tilde{f}}^p \, d\mu \leq \int_{X\setminus K} g_*^p\,d\mu+\epsilon&&
\end{flalign}
By Lemma \ref{lem:Sob}, $\tilde{f}\in N^{1,p}(X)$ with upper gradient $g$. Recall that $g_{\tilde{f}}$ is the minimial $p$-weak upper gradient and is smaller than any other upper gradient, i.e. $g_{\tilde{f}}\leq g$ (a.e.). By construction, $g_* \leq g$. Thus, the fourth property follows from Lemma \ref{lem:choice-ug} and the fact that $g\in L^p(X)$:

\[
\int_{X\setminus K} g_{\tilde{f}}^p d\mu \leq \int_{X\setminus K} g^p d\mu = \int_{X} g^p d\mu -\int_K g^p d\mu \leq \int_{X} g_*^p + \epsilon 2^{-4} -  \int_{K} g_*^p d\mu =\int_{X\setminus K} g_*^p + \epsilon 2^{-4}.
\]
\end{proof}

\subsection{Newton-Sobolev functions are quasicontinuous}
\begin{proposition}\label{prop:quasicont-complete} If $X$ is complete and separable and $f\in N^{1,p}(X)$, then $f$ is quasicontinuous.
\end{proposition}
We follow the arguments in \cite{bjornnages} and \cite{sha00}, but without relying on the hypothesis of properness and density of continuous functions in $N^{1,p}(X)$. Hence, we only provide a sketch.
\begin{proof}[Sketch of the proof of Proposition \ref{prop:quasicont-complete}]
By Theorem \ref{thm:density-complete}, there is a sequence
$f_i\in N^{1,p}(X)\cap C(X)$, for $i\in\N$,  with $\|f_i-f\|_{N^{1,p}(X)}\leq 2^{-i}$. 

Next, we apply the argument from the proof of \cite[Theorem 3.7]{sha00} to show that $f_i$ converges capacity almost everywhere to  $f\in N^{1,p}(X)$ . Fix $\epsilon_0>0$, and let $E_{\epsilon_0,n}=\{x\in X: |f_n-f|\geq \epsilon_0/n\}$. We have
\[
\int_X |f_n-f|^p d\mu \leq 2^{-np},
\]
and thus $\mu(E_{\epsilon_0,n})\leq n^p 2^{-np}\epsilon_0^{-p}$. Let $E_N = \bigcup_{n\geq N} E_{\epsilon, N}$. By a union bound, we get that for any $\epsilon>0$, there exists an $N$, so that $\mu(E_N)\leq \epsilon.$ The sequence of functions $f_n(x)$ converges uniformly to $f$ for any $x\in X\setminus E_N$, and thus for a.e. $x\in X$, since $\epsilon>0$ is arbitrary.

By considering $u_n=|f_n-f|n\epsilon_0^{-1}$ as a test function, we get $\Cp_p(E_{\epsilon_0,n})\leq n^p 2^{-np}\epsilon_0^{-p}$. At the expense of possibly increasing $N$, we get $\Cp_p(E_N)\leq \epsilon.$ Since $\epsilon>0$ is arbitrary, $\lim_{N\to \infty}\Cp_p(E_N)=0$. Further, $f_i$ converges pointwise to $f$ outside the set $E=\cap_{n=1}^\infty E_n$, which has capacity zero. 

Since the convergence is uniform, $f$ is continuous in $X\setminus E_N$, for every $N$. Therefore, $f$ is quasicontinuous, since $\lim_{N\to \infty}\Cp_p(E_N)=0$.
\end{proof}

\section{Localization and when $X$ is locally complete}
\label{sec:locallycomplete}
The previous section was focused entirely on complete spaces. 
In the final sections we improve these statements to a locally complete setting. Specifically, we prove Theorem \ref{thm:density-main}: Concluding that $N^{1,p}(X)\cap C(X)$ is dense in  $N^{1,p}(X)$ and that each function $f\in N^{1,p}(X)$ is quasicontinuous.

 These theorems will all be reduced to the complete setting by taking completions. This makes $X$ into an open set in its completion, and we are left to consider domains $\Omega$ in complete spaces. Then, in each case, we consider the set of points $X_\delta \subset X$, whose distance do the boundary in the completion is at least $\delta$, and construct partitions of unity subordinate to such sets. Each $X_\delta$ is complete, and the proofs mainly involve checking that we can ``patch'' together the information from each $X_\delta$ to their union, which is $X$.
 
 For technical reasons, we prove these theorems in a slightly different order from those in the complete setting.
 
 \subsection{Preliminaries on taking a completion}
 
 First, we address some measure theoretic issues in taking a completion. Let $X$ be locally complete, and let $\hat{X}$ be its completion. The completion is separable, if $X$ is separable. Further $X$ is an open subset of $\hat{X}$. If $\mu$ is a Radon measure on $X$, then we can define a Radon measure $\hat{\mu}$ on $\hat{X}$ as follows. If $E\subset X$ is Borel, then $E\cap X$ is also Borel and we can define $\hat{\mu}(E)=\mu(E\cap X)$. (In fact, by a different argument $E\cap X$ is Borel in $X$ whenever $E$ is Borel even when $X$ is not Borel measurable in $\hat{X}$, see \cite[Proof of Lemma 1]{saksman}). Since $\mu$ is finite on balls, so is $\hat{\mu}$ and therefore $\hat{\mu}$ is a Radon measure.  (See discussion at the beginning of Section \ref{sec:prelim}). 

Since we will be dealing with concepts relative to $X$ and $\hat{X}$ we need some care in our notation. For capacity, we will indicate the space $Y$ with respect to which it is computed in the superscript, as in $\Cp_p^Y(E)$, for $E\subset Y$. We remark, that if $E\subset X \subset Y$ and the measures on the spaces relate by restriction $\mu_X=\mu_Y|_X$ (where $X$ is measurable in $Y$), then $\Cp_p^X(E) \leq \Cp_p^Y(E)$. Here, we use the fact that in this same setting if $u\in N^{1,p}(Y)$, then $u|_X \in N^{1,p}(X)$, as readily follows from the definition.

\subsection{Quasicontinuity}

\begin{proof}[Proof of quasicontinuity in Theorem \ref{thm:density-main}] Fix $f\in N^{1,p}(X)$. Let $\hat{X}$ be the completion of $X$ and $\hat{\mu}$ be the extension of $\mu$ to $\hat{X}$.

We have that $X$ is also an open set in $\hat{X}$ since $X$ is locally complete. Let $\delta>0$ be arbitrary. Define $X_\delta = \{x: d(x,\hat{X}\setminus X) \geq \delta\}$. Then, $X_\delta$ is a closed subset of $\hat{X}$. Choose $\psi_\delta(x) = \min\{1,\frac{2}{\delta}d(x,\hat{X}\setminus X_{\delta/2})\}$. Then $\psi_\delta|_{X_\delta} \geq 1$,  $\psi_\delta$ is $2/\delta$-Lipschitz and $\psi_\delta |_{\hat{X}\setminus X_{\delta/2}}=0$.

Let $f_\delta = f \psi_\delta$. We have $f_\delta|_X \in N^{1,p}(X)$ and $f_\delta |_{\hat{X}\setminus X_{\delta/2}}=0 \in N^{1,p}(\hat{X}\setminus X_{\delta/2})$. Then, $f_\delta\in N^{1,p}(\hat{X})$ since $N^{1,p}(X)$ has the sheaf property: If $A,B\subset \hat{X}$ are open sets and $f|_A \in N^{1,p}(A), f|_B \in N^{1,p}(B)$, then $f|_{A\cup B}\in N^{1,p}(A\cup B)$.\footnote{This can be seen by the following argument: if $g_A,g_B\in L^p(A)$ are upper gradients for $f|_A$ and $f|_B$, then $g=g_A\ones_A + g_B \ones_B \in L^p(A\cup B)$ is an upper gradient of $f|_{A\cup B}$. Indeed, the upper gradient inequality \eqref{eq:ug} can be verified for any rectifiable curve $\gamma$ in $A\cup B$ by dividing it into finitely many parts contained in either $A$ or $B$.}

Then, by Proposition \ref{prop:quasicont-complete} we have that $f_\delta$ is quasicontinuous in $\hat{X}$. Therefore, for any $\delta>0$ there is an open subset $E_\delta$ so that $f_\delta |_{\hat{X}\setminus E_\delta}$ is continuous and $\Cp_p^{\hat{X}}(E_\delta)<\delta.$ Fix $\epsilon>0$ and let $E=\cup_{i=1}^\infty E_{\epsilon 2^{-i}}\cap X.$

We have $\Cp_p^X(E) \leq \Cp_p^{\hat{X}}(E)\leq \epsilon$.  Now, $f_{2^{-i}\epsilon}|_{X \setminus E}$ is continuous for every $i\in \N$. 
Therefore $f|_{X_{2^{-i}\epsilon} \setminus E}=f_{2^{-i}\epsilon}|_{X_{2^{-i}\epsilon}  \setminus E}$ is continuous on for any $i\in\N$. From this we get $f|_{X \setminus E}$ is continuous.
\end{proof}

\subsection{Density of continuous functions}

Here and in what follows, the support ${\rm supp}(f)$ of a function $f:X\to \R$ is the smallest closed set $C$ so that $f|_{X\setminus C}$ vanishes identically.  
In the proof of the density of continuous functions we will apply a partition of unity argument

We will need a standard construction for a partition of unity subordinate to a cover. Let $X_i = \{x : d(\hat{X}\setminus X, x)\geq 2^{-i}\}$ and $\Omega_i = \{x : d(\hat{X}\setminus X,x)>2^{-i}\}$.  In the following, the distance of a point to an empty set is defined as $\infty$. Also, we say that a sum of functions $\sum_{i=1}^\infty f_i(x)$ is \emph{locally finite} if for every $x\in X$ there exists a neighborhood, where only finitely many terms are non-zero.

\begin{lemma}\label{lem:partitionofunity} Let $\hat{X}$ be the completion of $X$ and let $X_i$ be defined as above.
For each $n\in \N$, There exist $4^{n}$-Lipschitz functions $\psi_n:\hat{X}\to [0,1]$, so that
\begin{enumerate}
\item[(a)] ${\mathrm{supp}(\psi_0)}\subset X_{1}$ and ${\mathrm{supp}(\psi_n)}\subset X_{n+1} \setminus X_{n-1}$ for $n\geq 1$;
\item[(b)]  the functions are a partition of unity: $\sum_{n=0}^\infty \psi_n(x)=1$ for $x\in X$; and
\item[(c)] the previous sum is locally finite in $X$: for every $x\in X$ there exists a $\delta>0$ so that are at most three $n\in\N$ so that $\psi_n(y)\neq 0$ for $y\in B(x,\delta)$.
\end{enumerate} 
\end{lemma}
\begin{proof}
Let $\psi_0(x) = \min\{1,2d(x,\hat{X}\setminus X_1)\}$.  Recursively, for $n\ge 1$, define 
\begin{equation}\label{eq:psin-def}
\psi_{n}(x)=\left(1-\sum_{k=0}^{n-1}\psi_{k}\right) \min\{1,2^{n+1}d(x,\hat{X}\setminus X_{n+1})\}.
\end{equation}
First, $\psi_0$ is $2$-Lipschitz, and by induction one can show that $\psi_n$ is Lipschitz with constant $(1+\cdots+4^{n-1})+2^{n+1}\leq 4^n$. 
We have $\psi_0|_{X_0}=1$. By induction, we get that $\sum_{k=0}^{n-1} \psi_{k}|_{X_{n-1}}=1$. Therefore, (b) holds. Moreover, this gives (a), since the first factor in (\ref{eq:psin-def}) vanishes on $X_{n-1}$ and the second factor vanishes outside $X_{n+1}$. 

Finally, we prove (c).  If $x\in X$, then $x\in X_n\setminus X_{n-1}$ for some $n\geq 0$, where $X_{-1}=\emptyset$ to simplify the argument. We have for $\delta=2^{-(n-1)}$ that $B(x,\delta)\subset X_{n+1}$. Thus, for $y\in B(x,\delta)$, due to (a), $\psi_k(y)\neq 0$ can only occur for $k=n-1,n,n+1$.
\end{proof}

We also need a fairly simple version of the sheaf property for Sobolev functions.

\begin{lemma}\label{lem:increasingsets}Let $p\in [1,\infty)$ and let $X$ be any metric measure space equipped with a Radon measure $\mu$, finite on balls. If $A_i\subset X$ is any increasing sequence of open sets, and $H: A:=\bigcup_{i=1}^\infty A_i \to [-\infty,\infty]$ is a function so that $H|_{A_i} \in N^{1,p}(A_i)$ with $\sup_{i\in \N} \|H\|_{N^{1,p}(A_i)}<\infty$, for $i\ge 1$, then $H\in N^{1,p}(A)$.
Further, $\|H\|_{N^{1,p}(A)}=\lim_{i\to\infty}\|H\|_{N^{1,p}(A_i)}$.
\end{lemma}
\begin{proof}
The $L^p$-version of the claim follows from monotone convergence, and we get $H\in L^{p}(A)$ and $\|H\|_{L^{p}(A)}=\lim_{i\to\infty}\|H\|_{L^{p}(A_i)}$

Let $g_i:X\to[0, \infty]$ be the zero-extension of the minimal $p$-weak upper gradient of $H|_{A_i}$. By locality of $p$-weak upper gradients, see \cite[Proposition 6.3.22]{shabook}, $g_i|_{A_j}=g_j$ almost everywhere on $A_j$ for all $j<i$. Thus, there exist a function $g:A \to [0,\infty]$ with $g\in L^p(A)$ and $g|_{A_i}=g_i$ almost everywhere on $A_i$. Thus, $g|_{A_i}$ is a $p$-weak upper gradient for $H|_{A_i}$ for all $i\in \N$. 

Fix $\epsilon>0$. By Lemma \ref{lem:lowersem}, for every $i$, we can find a lower semi-continuous $g_{i,\epsilon}:A_i \to [0,\infty]$ which is an upper gradient for $H|_{A_i}$ with $g_{i,\epsilon}\geq g|_{A_i}$ and $\|g_{i,\epsilon}-g|_{A_i}\|_{L^p(A_i)}\leq \epsilon 2^{-i}$. Extend $g_{i,\epsilon}$ by zero, and define $\tilde{g}=\sup_{i} g_{i,\epsilon}.$ We have, on the set $A$,
\[
|\tilde{g}-g|\le \sum_{i=1}^\infty |g_{i,\epsilon}-g|\ones_{A_i}.
\]
Then $\tilde{g}\in L^p(A)$, and, by monotone convergence,
\[
\|\tilde{g}\|_{L^p(A)}\leq \epsilon + \lim_{i\to\infty}\|g_i\|_{L^p(A_i)}.
\]

By construction, $\tilde{g}|_{A_i}\ge g_{i,\epsilon}$ and thus $\tilde{g}|_{A_i}$ is an upper gradient for $H|_{A_i}$. Every rectifiable curve in $\bigcup_{i\in \N}A_i$ is contained in $A_i$ for some $i\in \N$. This argument verifies \eqref{eq:ug} and $\tilde{g}$ is an upper gradient for $H$. Thus, $H\in N^{1,p}(A)$.  Further, 
\[
\|H\|_{N^{1,p}(A)}\leq (\|H\|_{L^p(A)}^p+\|\tilde{g}\|_{L^p(A)}^p)^{\frac{1}{p}} \leq \lim_{i\to\infty}((\epsilon + \|g_i\|_{L^p(A_i)})^p + \|H\|_{L^{p}(A_i)}^p)^{\frac{1}{p}}.
\]
Since $\epsilon>0$ is arbitrary the claim follows.
\end{proof}

\begin{proof}[Proof of density in Theorem \ref{thm:density-main}] Let $f\in N^{1,p}(X)$ be any function. Fix $\epsilon>0$. Let $\psi_n$ be the partition of unity functions from Lemma \ref{lem:partitionofunity}. We also define $\hat{\psi}_n=\psi_{n}+\psi_{n+1}+\psi_{n-1}$ for $n\geq 1$ and $\hat{\psi}_0=\psi_{0}+\psi_{1}$. For every $x\in\Omega$ we have finitely many $n$ so that $\hat{\psi}_n(x)\neq 0$. Further, whenever $\psi_n(x) \neq 0$, we have $\hat{\psi_n}(x)=1$. There are also constants $L_n$ so that $\hat{\psi}_n$ are $L_n$-Lipschitz. Indeed, with some care, we could show that $L_n\lesssim 4^n$, but we will not need this. 

As in the proof of the quasicontinuity in Theorem \ref{thm:density-complete} we set $f_n = f\hat{\psi}_n \in N^{1,p}(\hat{X})$. By Theorem \ref{thm:density-complete}, there is a continuous $u_n' \in N^{1,p}(\hat{X})\cap C(\hat{X})$ so that $\|f_n-u_n'\|_{N^{1,p}(\hat{X})}\leq \epsilon 2^{-4-n} (1+L_n)^{-1}$. Also, let $u_n=\hat{\psi}_n g_n'$. Then, by using the Leibniz rule (see \cite[Proposition 6.3.28]{shabook}), we get

\[
\|u_n-f_n\|_{N^{1,p}(\hat{X})} = \|\hat{\psi}_n(u_n-f_{n})\|_{N^{1,p}(\hat{X})}\leq 2(1+L_n)\|u_n-f_n\|_{N^{1,p}(\hat{X})} \leq 2^{-2-n}\epsilon.
\]

Let $u = \sum_{n=1}^\infty u_n$. Since the sum is a locally finite sum of continuous functions by Property (c) of Lemma \ref{lem:partitionofunity}, then $u\in C(X)$ and the sum is well defined.

We show that $u\in N^{1,p}(X)$. Fix an $i\in \N$. We have that $f|_{\Omega_i}=\sum_{i=0}^{n+1} f_i|_{\Omega_i}$ and that  $u|_{\Omega_i}=\sum_{i=0}^{n+1} u_i|_{\Omega_i}$. Thus,
\[
\|f-u\|_{N^{1,p}(\Omega_i)}\leq \sum_{n=1}^\infty\|f_n-u_n\|_{N^{1,p}(\Omega_i)} \leq \epsilon /2
\]
and $u|_{\Omega_i} \in N^{1,p}(\Omega_i)$ with a uniformly bounded norm independent of $i$.

By Lemma \ref{lem:increasingsets}, $u\in N^{1,p}(X)$. Further,  $\|g\|_{N^{1,p}(X)}=\lim_{i\to\infty} \|g\|_{N^{1,p}(\Omega_i)}$. Finally, by applying this argument to the difference $f-g$, we obtain $\|f-g\|_{N^{1,p}(X)}\leq \epsilon$. Since $g$ is continuous, the claim follows.
\end{proof}

\section{Choquet capacities and equivalence of definitions}\label{sec:choquet}

In the final section we study the capacity  $E\to \Cp_p(E)$ and condenser capacity $\Cp_p(E,F)$, and prove that they satisfy certain regularity properties. Specifically, we prove the following three theorems.

\begin{enumerate}
\item Theorem \ref{thm:capacityouterreg}: Concluding that $E\to \Cp_p(E)$ is outer regular.
\item Corollary \ref{cor:capacity-choquet}: Concluding that $E\to \Cp_p(E)$ is a Choquet capacity for $p>1$.
\item Theorem \ref{thm:capacity}: Concluding that different definitions of $\Cp_p(E,F)$ coincide. In particular, capacity can be computed with locally lipschitz funtions with locally lipschitz upper gradients.
\end{enumerate}

\subsection{Choquet capacity and outer regularity}

We start by defining a Choquet capacity. Denote by $\mathcal{P}(X)$ the collection of all subsets of $X$, i.e. its power set.

\begin{definition}\label{def:Choquet} A functional $I:\mathcal{P}(X) \to [0,\infty]$ is called a Choquet capacity, if it satisfies the following three properties.
\begin{enumerate}
\item Increasing: If $A\subset B \subset X$, then $I(A) \leq I(B)$.
\item Continuity from below: If $(A_n)_{n\in \N}$ is an increasing sequence of subsets of $X$, then
\[
\lim_{n\to\infty} I(A_n)=I\left(\bigcup_{n\in \N}A_n\right).
\] 
\item Continuity from above: If $(K_n)_{n\in \N}$ is a decreasing sequence of compact subsets of $X$, then 
\[
\lim_{n\to\infty} I(K_n)=I\left(\bigcap_{n\in \N}K_n\right).
\]
\end{enumerate}
\end{definition}

A reader interested in Choquet capacities may consult any of the following \cite{dellacherie,choquet}. A condensed treatise is available in \cite{carleson}. An earlier result showing that a variant of  $\Cp_p$, see Remark \ref{rmk:neighborhoodcapacity}, is Choquet is presented in \cite{kinnunenmartio}. One of the main motivations for introducing Choquet capacities is the ``Capacitability theorem'' of Choquet, which states that any analytic subset $A\subset X$ satisfies: $I(A)=\sup_{K\subset A} I(K)$, where the supremum is taken over compact subsets of $A$.

For $I=\Cp_p$, the increasing property is immediate from the definition. The continuity from below holds without further assumptions, when $p>1$. The continuity from above is reduced to the functional being \emph{outer-regular}. Recall that the functional $I$ is outer regular, if for every compact set $K\subset X$, and any $\epsilon>0$, there exists an open set $O$ such that $K\subset O$ and $I(O)\leq I(K)+\epsilon$. In other words, the main object is to establish outer regularity, and then collect all the pieces together to prove that $\Cp_p$ is a Choquet capacity.

We first prove the outer regularity of the capacity, which was stated in Theorem \ref{thm:capacityouterreg}. This is a repetition of the argument in \cite[Proof of Corollary 1.3]{bjornnages}, with the only change being that Proposition \ref{prop:quasicont-complete} is used instead of \cite[Theorem 1.1]{bjornnages}. For the reader's convenience, we sketch the idea here.
\begin{proof}[Sketch of proof of Theorem \ref{thm:capacityouterreg}] Let $u\in N^{1,p}(X)$ be any non-negative function with $u|_E \geq 1$. Fix $\epsilon>0$. Then, $u$ is quasicontinuous by Theorem \ref{thm:density-main}, and there is an open set $V$ with $\Cp_p(V)<\epsilon^p$ so that $u|_{X\setminus V}$ is continuous. Choose a non-negative function $v$ so that $\|v\|_{N^{1,p}(X)} \leq \epsilon$ and $v|_V \geq 1$.

 By continuity in $X\setminus V$,  there is an open set $O_E$ with $E\setminus V \subset O_E$  so that $u|_{O_E\cap X\setminus V} \geq 1-\epsilon$.  Consider the function $u_\epsilon = \frac{u}{1-\epsilon}+v$. Then $u_\epsilon|_{O_E\cup V}\geq 1$. The set $O_E \cup V$ is open, and thus,

$$\inf_{E\subset O} \Cp_p(O) \leq \|u_\epsilon\|^{p}_{N^{1,p}(X)} \leq \left(\frac{1}{(1-\epsilon)}\|u\|_{N^{1,p}(X)} + \epsilon\right)^p.$$

Taking an infimum over $u \in N^{1,p}(X)$ with $u|_{E} \geq 1$ and letting $\epsilon\rightarrow 0$ yields the claim.
\end{proof}

Next, we prove that $\Cp_p$ is a Choquet capacity. This was stated in the introduction as Corollary \ref{cor:capacity-choquet}.

\begin{proof}[Proof of Corollary \ref{cor:capacity-choquet}]
We verify the three properties of a Choquet capacity from Definition \ref{def:Choquet}.

\begin{enumerate}
\item \textbf{Increasing:} If $A\subset B \subset X$, then $\Cp_p(A)\leq \Cp_p(B)$, since every function $u\in N^{1,p}(X)$ with $u|_B=1$ also satisfies $u|_A=1$.

\item \textbf{Continuity from below:} We follow the proof of \cite{kinnunenmartio}, which is presented with a slightly different definition of capacity. Let $A_n$ be any increasing sequence of sets. By the increasing property, the property of continuity from below is automatic if $\lim_{n\to\infty} \Cp_p(A_n)=\infty$. Thus, we may assume that $\lim_{n\to\infty} \Cp_p(A_n)<\infty$. 

Choose any sequence $u_n \in N^{1,p}(X)$ so that $u_n|_{A_n}=1$ and 
\[
\lim_{n\to\infty} \|u_n\|_{N^{1,p}(X)}^p = \lim_{n\to\infty} \Cp_p(A_n).
\]
 The functions $u_n$ and their minimal $p$-weak upper gradients $g_{u_n}$ are uniformly bounded in $L^p(X)$. Therefore, by Mazur's Lemma, we may choose convex combinations $\tilde{u}_n$ of $\{u_k\}_{k=n}^\infty$ and corresponding convex combinations $\tilde{g}_n$ of $\{g_k\}_{k=n}^\infty$ so that $\tilde{u}_n$ and $\tilde{g}_n$ are Cauchy in $L^p(X)$ and so that $\tilde{g}_n$ is an upper gradient for $\tilde{u}_n$. Choose a subsequence $(n_k)_{k=1}^\infty$, with $n_k \geq k$, so that 

\begin{equation}\label{eq:Lp}
\sum_{k=1}^\infty |\tilde{u}_{n_{k+1}}-\tilde{u}_{n_{k}}| + |\tilde{g}_{n_{k+1}}-\tilde{g}_{n_{k}}| \in L^p(X). 
\end{equation}

Next, define $\tilde{u}_l = \sup_{k\geq l} \tilde{u}_{n_l}$ and $\tilde{g}_l = \sup_{k\geq l} \tilde{u}_{n_l}$. It follows from \eqref{eq:Lp} that $\tilde{u}_l, \tilde{g}_l \in L^p(X)$, and that as $l\to \infty$ they converge in $L^p(X)$. A fairly direct calculation using the definition \eqref{eq:ug} shows that $\tilde{g}_l$ is a $p$-weak upper gradient for $\tilde{u}_l$, and $\tilde{u}_l \in N^{1,p}(X)$.

Note that $\|\tilde{u}_l\|_{N^{1,p}(X)}^p\leq \|\tilde{u}_l\|_{L^p(X)}^p +  \|\tilde{g}_l\|_{L^p(X)}^p$. Then, by construction  and the $L^p(X)$ convergence, we get
\[
\lim_{l\to \infty} \|\tilde{u}_l\|_{L^p(X)}^p +  \|\tilde{g}_l\|_{L^p(X)}^p \leq \lim_{n\to\infty} \|u_n\|_{N^{1,p}(X)}^p = \lim_{n\to\infty} \Cp_p(A_n).
\]
Now, $\tilde{u}_l|_{A_k} \geq 1$ for every $k\geq l$. Thus $\tilde{u}_l|_{\bigcup_{k} A_k} \geq 1$. In particular, 
\[
\Cp_p\left(\bigcup_k A_k\right) \leq \|\tilde{u}_l\|_{L^p(X)}^p +  \|\tilde{g}_l\|_{L^p(X)}^p.
\]
Sending $l\to \infty$, gives 
\[
\Cp_p\left(\bigcup_k A_k\right) \leq \lim_{n\to\infty} \Cp_p(A_n).
\]
The opposite inequality follows from the increasing property. This completes the proof of continuity from below.

\item \textbf{Continuity from above:} Let $(K_n)_{n\in \N}$ be any decreasing sequence of compact sets and let $K= \bigcap_n K_n$. From the capacity being increasing, we get $\Cp_p(K) \leq \lim_{n\to\infty} \Cp_p(K_n)$. We next establish this inequality in the opposite direction. If $\Cp_p(K)=\infty$, then $\Cp_p(K)=\lim_{n\to\infty} \Cp_p(K_n)$. Thus, consider the case of $\Cp_p(K)<\infty$. By Theorem \ref{thm:capacityouterreg}, for every $\epsilon>0$, there exists an open set $O$ with $K\subset O$ and $\Cp_p(O)\leq \Cp_p(K)+\epsilon$. For $n$ sufficiently large $K_n \subset O$, and thus  by the increasing property, we get
\[
\Cp_p(K)\leq \lim_{n\to\infty} \Cp_p(K_n) \leq \Cp_p(K) +\epsilon.
\]
Since $\epsilon>0$ is arbitrary, the claim follows.
\end{enumerate}

\end{proof}

\subsection{Different definitions of capacity agree}

\begin{proof}[Proof of Theorem \ref{thm:capacity}]
Let $E,F\neq \emptyset$ be two closed, disjoint non-empty subsets in $X$ with $d(E,F)>0$.
It is straightforward to show that
\[
\Cp_p(E,F)\leq \Cp_p^c(E,F)\leq \Cp_p^{\lip}(E,F)\leq \Cp_p^{(\lip,\lip)}(E,F).
\]
Thus, it suffices to prove $\Cp_p^{(\lip,\lip)}(E,F)\leq \Cp_p(E,F)$. If $\Cp_p(E,F)=\infty$, this is obvious. Thus, assume $\Cp_p(E,F)<\infty$. Let $\epsilon>0$ be arbitrary. We can choose a function $u\in N^{1,p}(X)$ which is non-negative, with $u|_E=0$ and $u|_F=1$, and with an upper gradient $g_\epsilon$ such that
\[
\int g_\epsilon^p d\mu \leq \Cp_p(E,F) + \epsilon.
\]
Let $\hat{X}$ be the completion of $X$. Fix $x_0 \in X$. Extend $u$ and $g_\epsilon$ by zero to functions in  $L^p(\hat{X})$. 
Let $X_j = \{x\in X \cap \overline{B(x_0,j)} : d(x, \hat{X} \setminus X) \geq 2^{-j}\}$. Let $\psi_j$ be the partition of unity constructed in Lemma \ref{lem:partitionofunity}. Recall that $\sum_{n=0}^\infty \psi_n(x)=1$, each $\psi_n$ is $L_n$-Lipschitz for some $L_n<\infty$, and that ${\mathrm{supp}(\psi_n)}\subset X_{n+1} \setminus X_{n-1}$ for $n\geq 1$ and ${\rm supp}(\psi_0) \subset X_0$.

Let $E_j = E \cap X_j, F_j = F \cap X_j$. Let $j_0 \in \N$ be so that $E_j,F_j \neq \emptyset$ for all $j\geq j_0$.  
The space $X_j$ is complete and bounded, and so we can apply Theorem \ref{thm:capacity-complete} and Corollary \ref{rmk:refinedversion} to the functions $u|_{X_j}$ and $(g_\epsilon)|_{X_j}$. We obtain that for every $j\geq j_0$ there exist Lispchitz functions $u_j\in N^{1,p}(X_j)$ with locally Lipschitz upper gradients $g_j\in L^p(X_j)$, for $j\in \N$, so that  $u_j|_{E_j}=1, u_j|_{F_j}=0$ and $\|g_j-g_\epsilon\|_{L^p(X_j)}\leq 2^{-j}$ in $L^p(X_j)$. Extend each $u_j$ and $g_j$ by zero to an $L^p(X)$ function defined on all of $X$.  

We have $0\leq u_j\leq 1$. We will briefly consider the space $L^2(X_k)$ in order to avail ourselves of weak compactness in this space. The sets $X_k$ are bounded and have bounded measure for $k\in \N$. Thus, $u_j|_{X_k} \in L^2(X_k)$, for every $k$ and every $j\in \N$, and  $\sup_{j\in \N} \|u_j\|_{L^2(X_k)}<\infty$. Weak compactness allows us to take a subsequence converging weakly in $L^2(X_k)$. Further, by  Mazur's Lemma and a diagonal argument, we can take finite convex combinations $v_{j}$ of $\{u_j,u_{j+1}, \dots\}$ which converge in $L^2(X_k)$ for every $k\in \N$. It is direct to show that $v_j$ converges in $L^p(X_k)$ for every $k\in  \N$. 

Consider the corresponding convex combinations $h_j$ of the upper gradients in $\{g_j, g_{j+1}, \dots\}$. By using the definition of $g_j$, these converge, for every $k\in \N$, in $L^p(X_k)$ to $g_\epsilon|_{X_k}$, since $g_j|_{X_k}$ converge to $g_\epsilon|_{X_k}$ as $j\to\infty$.  By construction, $h_j$ is an upper gradient for $v_j$ in $X_j$, $v_j|_{E_k}=0, v_j|_{F_k}=1$ and $v_j$ is Lipschitz on $X_k$ for every $j\geq k$. Further, each $h_j$ is locally Lipschitz. 

Choose a subsequence $(n_k)_{k\in \N}$ so that $\|v_{n_k}-v_{n_{k+1}}\|_{L^p(X_{k+2})}\leq \epsilon L_k^{-1}2^{-k}$, $\|h_{n_k}-g_{\epsilon}\|_{L^p(X_{k+1})} \leq \epsilon 2^{-k}$ and so that $n_k \geq k+1$.
Finally, define 
\[
  U = \sum_{i=0}^\infty v_{n_i} \psi_i.
\]
By Property (c) of Lemma \ref{lem:partitionofunity}, the sum in $U$ is locally finite and $U$ is locally Lipschitz.  It is not hard to see that $U|_E=0$ and $U|_F=1$. 

Let $G:=\sum_{i=0}^\infty h_{n_i} \psi_i + L_i (\psi_{i-1}+\psi_{i}+\psi_{i+1})|v_{n_i}-U|$. We have the following estimates, since $\psi_i$ is a partition of unity:

\begin{align*}
\|G\|_{L^p(X)}&\leq \|g_\epsilon\|_{L^p(X)} + \|G-g_\epsilon\|_{L^p(X)} \\
&\leq \|g_\epsilon\|_{L^p(X)} + \sum_{i=0}^\infty \|(h_{n_i}-g_\epsilon)\psi_i + L_i \ones_{X_{i+2}}|v_{n_i}-U|\|_{L^p(X)} \\
&\leq\|g_\epsilon\|_{L^p(X)} + \sum_{i=1}^\infty \|(h_{n_{i}}-g_\epsilon)\|_{L^p(X_{i+1})} + L_i\|v_{n_i}-U\|_{L^p(X_{i+2})}\\
& \leq \Cp_p(E,F)^{1/p} + 4 \epsilon.
\end{align*}

Assume for the moment that $G$ is a $p$-weak upper gradient of $U$. Then, by Lemma \ref{lem:lowersem} for every $\epsilon>0$, there exists an upper gradient $g$ of $U$ with $\|g\|_{L^p(X)}\leq \|G\|_{L^p(X)}+\epsilon\leq \Cp_p(E,F)^{1/p} + 5 \epsilon.$ Since $\epsilon>0$ is arbitrary the claim follows. Thus, we only need to show that $G$ is a $p$-weak upper gradient of $U$. This is a matter of a final Lemma.

\begin{lemma} Suppose that $\psi_i$ are $L_i$ Lipschitz functions, so that ${\rm supp}\{\psi_i\}\subset X_i$, and so that $\sum_{i=1}^\infty \psi_i = 1$, where the sum is locally finite. Then, if $v_{n_i} \in N^{1,p}(X_{i+1})$ are functions with continuous upper gradients $h_{n_i} \in L^{p}(X_{i+1})$, then
\[G:=\sum_{i=0}^\infty h_{n_i} \psi_i + L_i (\psi_{i-1}+\psi_{i}+\psi_{i+1})|v_{n_i}-U|\]
is an upper gradient of
\[
  U = \sum_{i=0}^\infty v_{n_i} \psi_i.
\]
\end{lemma}

\begin{proof} Extend each $h_{n_i}, v_{n_i}$ by zero outside of $X_{i+1}$. This does not alter definitions of the functions $G$ and $U$ since ${\rm supp}(\psi_i)\subset X_{i+1}$. 
We have that $h_{n_i}$ is an upper gradient for $v_{n_i}$ in $X_{i+1}$. Note that $L_i (\psi_{i-1}+\psi_{i}+\psi_{i+1})$ is an upper gradient of $\psi_i$, since $\psi_i$ is $L_i$ Lipschitz, ${\rm supp}(\psi_i)\subset X_{i+1}\setminus X_{i-1}$ and $\ones_{X_{i+1}\setminus X_i} \leq (\psi_{i-1}+\psi_{i}+\psi_{i+1})$.  Thus, by the Leibnitz rule (see the proof of \cite[Proposition 6.3.28]{shabook}) we have that  $h_{n_i} \psi_i + L_i (\psi_{i-1}+\psi_{i}+\psi_{i+1})|v_{n_i}-U|$ is an upper gradient for $v_{n_i}\psi_i$ in $X_{i+1}$. Since $v_{n_i}\psi_i$ vanishes outside of $X_{i+1}$, $h_{n_i} \psi_i + L_i (\psi_{i-1}+\psi_{i}+\psi_{i+1})|v_{n_i}-U|$ is an upper gradient on all of $X$. Summing over $i\in\N$, we get that $G$ is a $p$-weak upper gradient for $U$. (It is direct to see that \eqref{eq:ug} is stable under countable sums.)
\end{proof}

\end{proof}

\bibliographystyle{acm}
\bibliography{pmodulus}

\begin{thebibliography}{10}

\bibitem{ambgigsav}
{\sc Ambrosio, L., Gigli, N., and Savar\'{e}, G.}
\newblock Density of {L}ipschitz functions and equivalence of weak gradients in
  metric measure spaces.
\newblock {\em Rev. Mat. Iberoam. 29}, 3 (2013), 969--996.

\bibitem{currents}
{\sc Ambrosio, L., and Kirchheim, B.}
\newblock Currents in metric spaces.
\newblock {\em Acta Math. 185}, 1 (2000), 1--80.

\bibitem{ambrosio}
{\sc Ambrosio, L., Pinamonti, A., and Speight, G.}
\newblock Weighted {S}obolev spaces on metric measure spaces.
\newblock {\em J. Reine Angew. Math. 746\/} (2019), 39--65.

\bibitem{basso2023geometric}
{\sc Basso, G., Marti, D., and Wenger, S.}
\newblock Geometric and analytic structures on metric spaces homeomorphic to a
  manifold.
\newblock {\em arXiv preprint arXiv:2303.13490\/} (2023).

\bibitem{bate}
{\sc Bate, D.}
\newblock Characterising rectifiable metric spaces using tangent spaces.
\newblock {\em Invent. Math. 230}, 3 (2022), 995--1070.

\bibitem{bjornbook}
{\sc Bj\"{o}rn, A., and Bj\"{o}rn, J.}
\newblock {\em Nonlinear potential theory on metric spaces}, vol.~17 of {\em
  EMS Tracts in Mathematics}.
\newblock European Mathematical Society (EMS), Z\"{u}rich, 2011.

\bibitem{bjornnages}
{\sc Bj\"{o}rn, A., Bj\"{o}rn, J., and Shanmugalingam, N.}
\newblock Quasicontinuity of {N}ewton-{S}obolev functions and density of
  {L}ipschitz functions on metric spaces.
\newblock {\em Houston J. Math. 34}, 4 (2008), 1197--1211.

\bibitem{bogachev07}
{\sc Bogachev, V.~I.}
\newblock {\em Measure theory. Vol. 2}.
\newblock Springer Science \& Business Media, 2007.

\bibitem{bourdonpi}
{\sc Bourdon, M., and Pajot, H.}
\newblock Poincar\'{e} inequalities and quasiconformal structure on the
  boundary of some hyperbolic buildings.
\newblock {\em Proc. Amer. Math. Soc. 127}, 8 (1999), 2315--2324.

\bibitem{capogna2022neumann}
{\sc Capogna, L., Kline, J., Korte, R., Shanmugalingam, N., and Snipes, M.}
\newblock Neumann problems for $ p $-harmonic functions, and induced nonlocal
  operators in metric measure spaces.
\newblock {\em arXiv preprint arXiv:2204.00571\/} (2022).

\bibitem{carleson}
{\sc Carleson, L.}
\newblock Selected problems on exceptional sets.
\newblock In {\em Selected reprints}, Wadsworth Math. Ser. Wadsworth, Belmont,
  CA, 1983, pp.~iv+100.

\bibitem{che99}
{\sc Cheeger, J.}
\newblock {Differentiability of {L}ipschitz functions on metric measure
  spaces}.
\newblock {\em Geom. Funct. Anal. 9}, 3 (1999), 428--517.

\bibitem{cheeger1997structure}
{\sc Cheeger, J., and Colding, T.~H.}
\newblock On the structure of spaces with ricci curvature bounded below. i.
\newblock {\em Journal of Differential Geometry 46}, 3 (1997), 406--480.

\bibitem{choquet}
{\sc Choquet, G.}
\newblock Forme abstraite du t\'{e}or\`eme de capacitabilit\'{e}.
\newblock {\em Ann. Inst. Fourier (Grenoble) 9\/} (1959), 83--89.

\bibitem{dellacherie}
{\sc Dellacherie, C.}
\newblock {\em Ensembles analytiques, capacit\'{e}s, mesures de {H}ausdorff}.
\newblock Lecture Notes in Mathematics, Vol. 295. Springer-Verlag, Berlin-New
  York, 1972.

\bibitem{seb2020}
{\sc Eriksson-Bique, S.}
\newblock Density of lipschitz functions in energy.
\newblock {\em Preprint (arXiv:2012.01892)\/} (2020).

\bibitem{duality}
{\sc Eriksson-Bique, S., and Poggi-Corradini, S.}
\newblock On the sharp lower bound for duality of modulus, 2021.
\newblock Submitted and available as (2102.03035).

\bibitem{teriseb}
{\sc Eriksson-Bique, S., and Soultanis, E.}
\newblock Curvewise characterizations of minimal upper gradients and the
  construction of a {S}obolev differential.
\newblock {\em To appear in {A}nalysis and PDE\/} (2021).
\newblock arXiv:2102.08097.

\bibitem{gromov2014plateau}
{\sc Gromov, M.}
\newblock Plateau-stein manifolds.
\newblock {\em Open Mathematics 12}, 7 (2014), 923--951.

\bibitem{HKM06}
{\sc Heinonen, J., Kilpel\"{a}inen, T., and Martio, O.}
\newblock {\em Nonlinear potential theory of degenerate elliptic equations}.
\newblock Dover Publications, Inc., Mineola, NY, 2006.
\newblock Unabridged republication of the 1993 original.

\bibitem{heinonenkoskela}
{\sc Heinonen, J., and Koskela, P.}
\newblock Quasiconformal maps in metric spaces with controlled geometry.
\newblock {\em Acta Math. 181}, 1 (1998), 1--61.

\bibitem{shabook}
{\sc Heinonen, J., Koskela, P., Shanmugalingam, N., and Tyson, J.~T.}
\newblock {\em Sobolev spaces on metric measure spaces}, vol.~27 of {\em New
  Mathematical Monographs}.
\newblock Cambridge University Press, Cambridge, 2015.

\bibitem{keith03}
{\sc Keith, S.}
\newblock {Modulus and the Poincar\'e inequality on metric measure spaces}.
\newblock {\em Math. Z. 245\/} (2003), 255–292.

\bibitem{kinnunenmartio}
{\sc Kinnunen, J., and Martio, O.}
\newblock Choquet property for the {S}obolev capacity in metric spaces.
\newblock In {\em Proceedings on {A}nalysis and {G}eometry ({R}ussian)
  ({N}ovosibirsk {A}kademgorodok, 1999)\/} (2000), Izdat. Ross. Akad. Nauk Sib.
  Otd. Inst. Mat., Novosibirsk, pp.~285--290.

\bibitem{lahtisha}
{\sc Lahti, P., and Shanmugalingam, N.}
\newblock Trace theorems for functions of bounded variation in metric spaces.
\newblock {\em J. Funct. Anal. 274}, 10 (2018), 2754--2791.

\bibitem{le2018primer}
{\sc Le~Donne, E.}
\newblock A primer on carnot groups: homogenous groups, carnot-carath{\'e}odory
  spaces, and regularity of their isometries.
\newblock {\em Analysis and Geometry in Metric Spaces 5}, 1 (2018), 116--137.

\bibitem{milman97}
{\sc Milman, V.~A.}
\newblock Extension of functions preserving the modulus of continuity.
\newblock {\em Mathematical Notes 61}, 2 (1997), 193--200.

\bibitem{ntalampekos2023polyhedral}
{\sc Ntalampekos, D., and Romney, M.}
\newblock Polyhedral approximation of metric surfaces and applications to
  uniformization.
\newblock {\em Duke Mathematical Journal 1}, 1 (2023), 1--62.

\bibitem{rajala}
{\sc Rajala, K.}
\newblock Uniformization of two-dimensional metric surfaces.
\newblock {\em Invent. Math. 207}, 3 (2017), 1301--1375.

\bibitem{walterreal}
{\sc Rudin, W.}
\newblock {\em Principles of mathematical analysis}, third~ed.
\newblock McGraw-Hill Book Co., New York-Auckland-D\"{u}sseldorf, 1976.
\newblock International Series in Pure and Applied Mathematics.

\bibitem{saksman}
{\sc Saksman, E.}
\newblock Remarks on the nonexistence of doubling measures.
\newblock {\em Ann. Acad. Sci. Fenn. Math. 24\/} (1999), 155--164.

\bibitem{sha00}
{\sc Shanmugalingam, N.}
\newblock {Newtonian spaces: an extension of {S}obolev spaces to metric measure
  spaces}.
\newblock {\em Rev. Mat. Iberoamericana 16}, 2 (2000), 243--279.

\bibitem{shaharmonic}
{\sc Shanmugalingam, N.}
\newblock Harmonic functions on metric spaces.
\newblock {\em Illinois J. Math. 45}, 3 (2001), 1021--1050.

\bibitem{sodini}
{\sc Sodini, G.~E.}
\newblock The general class of {W}asserstein {S}obolev spaces: density of
  cylinder functions, reflexivity, uniform convexity and {C}larkson's
  inequalities.
\newblock {\em Calc. Var. Partial Differential Equations 62}, 7 (2023), Paper
  No. 212, 41.

\bibitem{sormani2011intrinsic}
{\sc Sormani, C., and Wenger, S.}
\newblock The intrinsic flat distance between riemannian manifolds and other
  integral current spaces.
\newblock {\em Journal of Differential Geometry 87}, 1 (2011), 117--199.

\bibitem{guofangwei}
{\sc Wei, G.}
\newblock Manifolds with a lower {R}icci curvature bound.
\newblock In {\em Surveys in differential geometry. {V}ol. {XI}}, vol.~11 of
  {\em Surv. Differ. Geom.} Int. Press, Somerville, MA, 2007, pp.~203--227.

\end{thebibliography}

\end{document}